\documentclass{siamltex}
\usepackage{amsmath,latexsym,amsfonts,amssymb,mathrsfs,verbatim,euscript}

\usepackage{algpseudocode,algorithm}
\floatname{algorithm}{Pseudocode}

\newtheorem{assumption}{Assumption}
\newtheorem{remark}{Remark}[section]

\newcommand{\beqn}{\begin{equation}}
\newcommand{\eeqn}{\end{equation}}

\newcommand{\fa}{{\mathfrak a}}

\def\R{\mathbb{R}}
\def\N{\mathbb{N}}

\newcommand{\tru}{\operatorname{tr}}

\newcommand{\Gc}{\mathscr{G}}

\newcommand{\calD}{\mathcal{D}}

\newcommand{\cK}{\mathcal{K}}

\newcommand{\cS}{{U}}
\newcommand{\cX}{\mathcal{X}}
\newcommand{\bcX}{\mathcal{X}}
\newcommand{\cY}{\mathcal{Y}}
\newcommand{\bcY}{\mathcal{Y}}

\newcommand{\satop}[2]{\stackrel{\scriptstyle{#1}}{\scriptstyle{#2}}}

\newcommand{\bsalpha}{{\boldsymbol{\alpha}}}
\newcommand{\bsbeta}{\boldsymbol{\beta}}
\newcommand{\bsgamma}{{\boldsymbol{\gamma}}}
\newcommand{\bssigma}{\boldsymbol{\sigma}}

\newcommand{\bstau}{\boldsymbol{\tau}}

\newcommand{\bsnu}{{\boldsymbol{\nu}}}
\newcommand{\bsell}{{\boldsymbol{\ell}}}

\newcommand{\bsq}{{\boldsymbol{q}}}
\newcommand{\bsk}{{\boldsymbol{k}}}
\newcommand{\bsl}{\boldsymbol{l}}

\newcommand{\bsx}{{\boldsymbol{x}}}
\newcommand{\bsy}{{\boldsymbol{y}}}

\newcommand{\bsE}{\boldsymbol{E}}
\newcommand{\bsU}{\boldsymbol{U}}
\newcommand{\bsV}{\boldsymbol{V}}
\newcommand{\bsW}{\boldsymbol{W}}
\newcommand{\bsX}{\boldsymbol{X}}
\newcommand{\bsY}{\boldsymbol{Y}}
\newcommand{\half}{{\textstyle\frac{1}{2}}}
\newcommand{\bhalf}{{\textstyle\boldsymbol{\frac{1}{2}}}}
\newcommand{\rd}{\mathrm{d}}
\newcommand{\bbR}{\mathbb{R}}

\newcommand{\bbZ}{\mathbb{Z}}
\newcommand{\bbN}{\mathbb{N}}

\newcommand{\calO}{\mathcal{O}}
\newcommand{\calW}{\mathcal{W}}
\newcommand{\cL}{\mathcal{L}}
\newcommand{\setu}{\mathrm{\mathfrak{u}}}
\newcommand{\setv}{\mathrm{\mathfrak{v}}}

\newcommand{\KL}{Karhunen-Lo\`eve }

\newcommand{\setD}{{\mathfrak{D}}}
\newcommand{\mask}[1]{{}}
\newcommand{\bszero}{{\boldsymbol{0}}}

\newcommand{\bsOmega}{{\boldsymbol{\Omega}}}

\newcommand{\bsone}{{\boldsymbol{1}}}
\newcommand{\wal}{\operatorname{wal}}

\newcommand{\Z}{\mathbb{Z}}

\newcommand{\calS}{\mathcal S}

\newcommand{\be}{\begin{equation}}
\newcommand{\ee}{\end{equation}}
\newcommand{\ba}{\begin{array}}
\newcommand{\ea}{\end{array}}
\newcommand{\beas}{\begin{eqnarray*}}
\newcommand{\eeas}{\end{eqnarray*}}
\newcommand{\bea}{\begin{eqnarray}}
\newcommand{\eea}{\end{eqnarray}}

\title{Higher order QMC Petrov-Galerkin discretization for affine parametric operator equations
with random field inputs\footnote{\today}}

\author{Josef Dick\footnotemark[2]
 \and Frances Y. Kuo\footnotemark[2]
 \and Quoc T. Le Gia\footnotemark[2]
 \and Dirk Nuyens\footnotemark[3]
 \and Christoph Schwab\footnotemark[4]}

\begin{document}

\maketitle

\renewcommand{\thefootnote}{\fnsymbol{footnote}}

 \footnotetext[2]{School of Mathematics and Statistics,
                  University of New South Wales, Sydney NSW 2052, Australia
                  ({\tt josef.dick@unsw.edu.au}, {\tt f.kuo@unsw.edu.au}, {\tt
                  qlegia@unsw.edu.au}). The work of these authors was
                  supported by Australian Research Council Discovery
                  Projects. The first author's work was additionally
                  supported by an Australian Research Council QEII
                  Fellowship.
 }
 \footnotetext[3]{Department of Computer Science, KU Leuven, Belgium
                  ({\tt dirk.nuyens@cs.kuleuven.be}).
                  This author is a fellow of the Research Foundation Flanders (FWO).
 }
 \footnotetext[4]{Seminar for Applied Mathematics, ETH Z\"urich, ETH Zentrum,
                  HG G57.1, CH8092 Z\"urich, Switzerland ({\tt christoph.schwab@sam.math.ethz.ch}).
                  This author's work was supported by the European Research Council under grant AdG247277.}

\renewcommand{\thefootnote}{\arabic{footnote}}

\begin{abstract}
We construct quasi-Monte Carlo methods to approximate the expected values
of linear functionals of Petrov-Galerkin discretizations of
parametric operator equations which depend on a possibly infinite sequence
of parameters. Such problems arise in the numerical solution of
differential and integral equations with random field inputs. We analyze
the regularity of the solutions with respect to the parameters in terms of
the rate of decay of the fluctuations of the input field. If $p\in (0,1]$
denotes the ``summability exponent'' corresponding to the fluctuations in
affine-parametric families of operators, then we prove that deterministic
``interlaced polynomial lattice rules'' of order $\alpha = \lfloor 1/p
\rfloor+1$ in $s$ dimensions with $N$ points can be constructed using a
fast component-by-component algorithm, in $\calO(\alpha\,s\, N\log N +
\alpha^2\,s^2 N)$ operations, to achieve a convergence rate of
$\calO(N^{-1/p})$, with the implied constant independent of $s$. This
dimension-independent convergence rate is superior to the rate
$\calO(N^{-1/p+1/2})$ for $2/3\leq p\leq 1$, which was recently
established for randomly shifted lattice rules under comparable
assumptions. In our analysis we use a non-standard Banach space setting
and introduce ``smoothness-driven product and order dependent (SPOD)''
weights for which we develop a new fast CBC construction.
\end{abstract}



\pagestyle{myheadings} \thispagestyle{plain}

\markboth{J.~DICK, F.~Y. KUO, Q.~T.~LE GIA, D.~NUYENS AND
CH.~SCHWAB}{HIGHER ORDER QMC PETROV-GALERKIN DISCRETIZATION}
\section{Introduction}
The efficient numerical computation of statistical quantities for
solutions of partial differential and of integral equations with random
inputs is a key task in uncertainty quantification in engineering and in
the sciences. The quantity of interest is expressed as a mathematical
expectation, and the efficient computation of these quantities involves
two basic steps: i) the approximate (numerical) solution of the operator
equation, and ii) the approximate evaluation of the mathematical
expectation by numerical integration. In the present paper, we outline a
strategy towards these two aims which is based on i)
\emph{Petrov-Galerkin discretization} of the operator equation and on
ii) \emph{Quasi-Monte Carlo} (QMC) integration.

The present paper is motivated in part by \cite{KSS12}, where QMC
integration using a family of \emph{randomly shifted lattice rules} was
combined with Finite Element discretization for a model parametric
diffusion equation, and in part by \cite{ScMCQMC12}, where the methodology
was extended to an abstract family of parametric operator equations. In
this paper, we follow the methodology of \cite{KSS12} in the abstract
setting of \cite{ScMCQMC12}, but in contrast to \cite{KSS12,ScMCQMC12},
we use \emph{deterministic, ``interlaced polynomial lattice rules''},
which provide a convergence rate beyond order one for smooth integrands;
whereas order one was the limitation in \cite{KSS12,ScMCQMC12}.

Contrary to Monte Carlo methods which require uniformly distributed
samples of random input functions, QMC (and other) quadrature methods
require the introduction of coordinates of integration prior to numerical
quadrature. In the context of random field inputs with non-degenerate
covariance operators, a \emph{countable} number of coordinates is required
to describe the random input data, e.g., by a \KL expansion. Therefore, in
the present work, we consider in particular that the operator equation
contains not only a finite number of random input parameters, but rather
depends on \emph{random field inputs}, i.e., it contains random functions
of space and, in evolution problems,
of time which describe uncertainty in the problem under consideration.

More precisely, let $\bsy := (y_j)_{j\ge 1}$ denote the possibly
countable set of parameters from a domain $U \subseteq \R^\N$, and let
$A(\bsy)$ denote a $\bsy$-parametric bounded linear operator between
suitably defined spaces $\bcX$ and $\bcY'$. Then we wish to solve the
following parametric operator equation: given $f\in \bcY'$, for every
$\bsy\in U$ find $u(\bsy)\in \bcX$ such that
\begin{equation}\label{eq:main}
  A(\bsy)\, u (\bsy) = f \;.
\end{equation}
Such parametric operator equations arise from partial differential
equations with random field input, see, e.g.,
\cite{SchwabGittelsonActNum11}.
We assume in this paper the simplest
case, namely, that $A(\bsy)$ has ``\emph{affine}''
\emph{parameter dependence}, i.e.,
there exists a sequence
$\{ A_j\}_{j\geq 0} \subset \cL(\bcX,\bcY')$ such that for every $\bsy \in U$ we can write
\begin{equation}\label{eq:Baffine}
  A(\bsy) = A_0 + \sum_{j\ge 1} y_j\, A_j \;.
\end{equation}
A concrete example is the diffusion problem considered, e.g., in
\cite{KSS12}, in which the diffusion in random media is modeled by equation
\eqref{eq:main} with $A(\bsy) = -\nabla \cdot (a(\bsy)\nabla)$, and where
the diffusion coefficients are expanded in terms of a \KL expansion
$a(\bsy) = \bar{a} + \sum_{j \ge 1} y_j\, \psi_j$, leading to $A_0 =
-\nabla \cdot ( \bar{a} \nabla)$ and $A_j = -\nabla \cdot (\psi_j\nabla)$.

Here, as in \cite{KSS12}, we restrict ourselves to the
(infinite-dimensional) parameter domain
\[
  U \,=\, [-\half,\half]^\bbN\;.
\]
Some assumptions on the ``\emph{nominal}'' (or ``\emph{mean field}'')
operator $A_0$ and the ``\emph{fluctuation}'' operators $A_j$ are required
to ensure that the sum in \eqref{eq:Baffine} converges, and to ensure
existence and uniqueness of the solution $u(\bsy)$ in \eqref{eq:main} for
all $\bsy\in U$; these will be given in \S\ref{sec:pre}. In addition, we
shall consider the parametric Petrov-Galerkin approximation
$u^h(\bsy) \in \cX_h\subset \cX$, to be defined in \eqref{eq:parmOpEqh}
below, as well as $u^h_s(\bsy)$, corresponding to the Petrov-Galerkin
approximation of the problem with the sum in \eqref{eq:Baffine} truncated
to $s$ terms (this is equivalent to setting $y_j = 0$ for $j>s$). Further
assumptions on $A_0$ and $A_j$ are required for our regularity and
approximation results; these will all be given in \S\ref{sec:pre}. For now
we mention only one key assumption, namely, that there exists $p\in (0,1]$
for which
\begin{equation} \label{eq:psumpsi0}
  \sum_{j\ge 1} \| A_j\|_{\cL(\bcX,\bcY')}^p \,<\, \infty\;,
\end{equation}
where $\|\cdot\|_{\cL(\bcX,\bcY')}$ denotes the operator norm for the set
of all bounded linear mappings from $\bcX$ to $\bcY'$. This assumption
implies a decay of the fluctuation coefficients $A_j$, with stronger decay
as the value of $p\in (0,1]$ decreases.

For a given bounded linear functional $G(\cdot): \cX \to \bbR$, we are
interested in computing expected values of $G(u(\bsy))$ with respect to
$\bsy\in U$, i.e., an integral of the functional $G(\cdot)$ of the
parametric solution,
\begin{equation} \label{eq:int}
 I(G(u)) \,:=\, \int_U G(u(\bsy)) \,\rd\bsy\;,
\end{equation}
over the infinite dimensional domain of integration $U$.
We truncate the infinite sum in \eqref{eq:Baffine} to $s$ terms and solve
the corresponding operator equation \eqref{eq:main} using
Petrov-Galerkin discretization from a dense, one-parameter family
$\{\cX^h\}$ of subspaces of $\cX$. Denoting this dimension-truncated
Petrov-Galerkin solution by $u^h_s$, we then approximate the
corresponding $s$-dimensional integral using QMC quadrature,
\begin{equation} \label{eq:qmcG}
  \frac{1}{N} \sum_{n=0}^{N-1} G\big(u^h_s\big(\bsy_n - \bhalf\big)\big)\;,
\end{equation}
where $\bsy_0,\ldots,\bsy_{N-1}\in [0,1]^s$ denote $N$ points from a
properly chosen QMC rule, and the shift of coordinates by $\bhalf$ takes
care of the translation from $[0,1]^s$ to $[-\half,\half]^s$. Note that
each evaluation of the integrand at a single QMC point $\bsy_n$ requires
the approximate (Petrov-Galerkin) solution of one operator equation
for $u^h_s(\bsy_n)$.

There are three sources of error in approximating \eqref{eq:int} by
\eqref{eq:qmcG}: a \emph{Galerkin discretization error} depending on~$h$,
a \emph{dimension truncation error} depending on~$s$, and a \emph{QMC
quadrature error} depending on~$N$. The main focus of this paper will be on the analysis of the QMC error: we
prove that \emph{interlaced polynomial lattice rules} \cite{Go13,DiGo12}
can be constructed using a
component-by-component (CBC) algorithm to achieve a rate of convergence of
\[
  \calO(N^{-1/p})\;,
\]
with $p \in (0,1]$ as in \eqref{eq:psumpsi0}, and with the implied
constant independent of $N$, $h$, and $s$, but dependent on $p$. In fact,
the constant grows exponentially in $1/p^2$, thus the constant is large
for small values of $p$.

The function space setting for QMC integration considered in this paper
uses a Banach space norm with two parameters $1\le q\le\infty$ and $1\le
r\le\infty$, corresponding to the $L_q$ norm of functions and an $\ell_r$
norm of vectors combining these $L_q$ norms. Often $q$ and $r$ are taken
to be the same value in the literature, with $q=r=2$ giving the Hilbert
space setting. However, as discussed in \cite{KSS11}, decoupling $q$ and
$r$ allows more flexibility in the analysis, since the two parameters play
different roles. (The $L_q$ norm increases with increasing $q$, while the
$\ell_r$ norm increases with decreasing $r$.)
The results in \cite{KSS12,ScMCQMC12} are based on the Hilbert space
setting, with a convergence rate of
$\calO(N^{-\min(1/p-1/2,1-\delta)})$, for any $\delta>0$, which is capped at
order one. The main result of this paper is based on $r=\infty$ and it
holds \emph{for all values of $q$}. The convergence rate of
$\calO(N^{-1/p})$ obtained in this paper is an improvement by a factor of
$N^{-1/2}$ for $2/3 < p \le 1$ and by a factor of $N^{-1/p+1}$ for $0 < p
< 2/3$. The former improvement is due to our switch to a non-Hilbert
space setting. The latter improvement is due to the use of higher order
QMC rules.

Put differently, as discussed in \cite[p.\ 3368]{KSS12}, to achieve
nearly order one convergence rate the randomly shifted lattice rules
considered in \cite{KSS12,ScMCQMC12} require $p\le 2/3$; other lattice
rules require $p\le 1/2$; Niederreiter and Sobol$'$ sequences require
$p\le 1/3$. On the other hand, the interlaced polynomial lattice rules
considered in this paper give order one convergence rate already when
$p=1$.

In comparison, under the same assumption \eqref{eq:psumpsi0}, the paper
\cite{CDS1} establishes $p$-summability of generalized (Legendre)
polynomial chaos expansions of the integrand $G(u(\cdot))$ in
\eqref{eq:int}, and shows that $N$-term approximation of the integrand has
a convergence rate in $L_2$ norm of $\calO(N^{-1/p + 1/2)})$, with the
implied constant independent of the dimension of the integration domain.
This rate could be realized, for example, by adaptive Galerkin
projections. It also suggests an $N$-term approximation rate in the
(natural for integration) $L_1$ norm of $\calO(N^{-1/p})$.

Our QMC quadrature approach requires the use of interlaced polynomial
lattice rules of order $\alpha = \lfloor 1/p\rfloor + 1$, which is at
least $2$. (Thus we cannot prove our results using classical QMC
rules, as for instance described in \cite{Nie92}, since those are of
order~$1$.) Similar to the analysis in \cite{KSS12,ScMCQMC12}, we need to
choose ``\emph{weights}'' for the function space setting to ensure that
the implied constant for the convergence rate is bounded independently of
the truncation dimension $s$. The regularity analysis reveals the need to
use weights that are \emph{not} of ``POD'' form (namely, ``product and
order dependent'' form) as in \cite{KSS12}, but of a more general form
which we call ``SPOD weights", for ``smoothness-driven product and order
dependent'' weights,
see \eqref{eq_defSPOD} ahead. For these SPOD weights, we develop a new \emph{fast component-by-component construction} of interlaced polynomial lattice rules, with cost of $\calO(\alpha\,s\, N\log N + \alpha^2\,s^2 N)$ operations.

The outline of this paper is as follows. In \S\ref{sec:pre} we present a
class of parametric operator equations, review the parametric and spatial
regularities of their solutions, give a synopsis of the
Petrov-Galerkin discretization of these equations, and outline some
estimates relating to dimension truncation. In \S\ref{Sc:HiOrdQMC} we
derive a worst case error bound for digital nets in a novel weighted
Banach space setting, prove that interlaced polynomial lattice rules can
be constructed by a CBC algorithm to achieve a dimension-independent error
bound with a good convergence rate, and explain how to implement the
algorithm in an efficient way. Finally in \S\ref{sec:comb-err} we summarize the combined QMC
Petrov-Galerkin error bound.
\section{Problem formulation}
\label{sec:pre}
Generalizing results of \cite{CDS1}, we study well-posedness, regularity
and polynomial approximation of solutions for a family of abstract
parametric saddle point problems, with operators depending on a sequence of parameters. The results cover a wide range of operator equations: among them are (stationary and
time-dependent) diffusion in random media \cite{CDS1},
wave propagation \cite{HoSc12Multi},
parametric, nonlinear PDEs \cite{CCS2}
and optimal control problems for uncertain systems \cite{KunothCS2011}.
%
\subsection{Parametric operator equations}
\label{ssec:affparops}
We denote by $\bcX$ and $\bcY$ two separable and reflexive Banach spaces
over $\mathbb{R}$ (all results will hold with the obvious modifications
also for spaces over $\mathbb{C}$) with (topological) duals $\bcX'$ and
$\bcY'$, respectively. By $\cL(\bcX,\bcY')$, we denote the set of bounded
linear operators $A:\bcX \to\bcY'$.

As we explained in the introduction, let $\bsy := (y_j)_{j \geq 1} \in
U = [-\half,\half]^\bbN$ be a countable set of parameters.
For every $f\in \bcY'$ and every $\bsy\in U$,
we wish to solve the parametric operator
equation \eqref{eq:main}, where the operator $A(\bsy)\in\cL(\bcX,\bcY')$
is of affine parameter dependence, see \eqref{eq:Baffine}.
In order for
the sum in \eqref{eq:Baffine} to converge, we impose the following
assumptions on the sequence $\{A_j\}_{j\geq 0}\subset \cL(\cX,\cY')$.
In
doing so, we associate with the operator $A_j$ the bilinear forms
$\fa_j(\cdot,\cdot):\bcX\times \bcY \rightarrow \mathbb{R}$ via
$$
  \forall v\in \cX,\;w\in \cY:\quad
  \fa_j(v,w) \,=\, {_{\cY}}\langle w, A_j v \rangle_{\cY'}\;,
  \quad j=0,1,2,\ldots
  \;.
$$
\begin{assumption}\label{ass:AssBj}
The sequence $\{ A_j \}_{j\geq 0}$ in \eqref{eq:Baffine} satisfies
the following conditions:
\begin{enumerate}
\item%
The \emph{nominal operator} $A_0\in \cL(\bcX,\bcY')$ is boundedly invertible,
i.e., there exists $\mu_0 > 0$ such that (cf. the inf-sup conditions
in \cite{BF})
\begin{equation}\label{eq:B0infsup} 
 \inf_{0\ne v \in \bcX} \sup_{0\ne w \in \bcY}
 \frac{\fa_0(v,w)}{\| v \|_{\bcX} \|w\|_{\bcY}}
 \ge \mu_0\;,\quad
 \inf_{0\ne w \in \bcY} \sup_{0\ne v \in \bcX}
 \frac{\fa_0(v,w)}{\| v \|_{\bcX} \|w\|_{\bcY}}
 \ge \mu_0
 \;.
\end{equation}
\item%
The \emph{fluctuation operators} $\{ A_j \}_{j\geq 1}$ are small with respect
to $A_0$ in the following sense: there exists a constant $0 < \kappa <
2$ such that 
\begin{equation} \label{eq:Bjsmall} 
 \sum_{j\geq 1} \beta_j \leq \kappa < 2\;,
 \quad\mbox{where}\quad
 \beta_j \,:=\, \| A_0^{-1} A_j \|_{\cL(\cX,\cY')}\;,
 \quad j=1,2,\ldots
 \;.
\end{equation}
\end{enumerate}
\end{assumption}
Assumption~\ref{ass:AssBj} is sufficient for the bounded invertibility of
$A(\bsy)$, uniformly with respect to the parameter sequence $\bsy\in U$.
(This corresponds to the assumption of the uniform bound on the random
coefficient of the elliptic PDE considered in \cite{KSS12}.)
\begin{theorem}\label{thm:BsigmaInv}
Under Assumption~\ref{ass:AssBj}, for every realization $\bsy\in \cS$ of
the parameter vector, the affine parametric operator $A(\bsy)$ given by
\eqref{eq:Baffine} is boundedly invertible. Specifically, for the bilinear
form $\fa(\bsy;\cdot,\cdot): \bcX\times\bcY\to\R$ associated with
$A(\bsy)\in \cL(\bcX,\bcY')$ via
\begin{equation}\label{eq:parmbil}
 \fa(\bsy;v,w) \,:=\, {_{\cY}}\langle w, A(\bsy) v\rangle_{\cY'}\;,
\end{equation}
there hold the uniform (with respect to $\bsy\in \cS$) inf-sup conditions
with $\mu = (1 - \kappa/2)\mu_0$,
\begin{equation}\label{eq:Bsiinfsup}
 \forall \bsy \in \cS:
 \quad
 \inf_{0\ne v \in \bcX} \sup_{0\ne w \in \bcY}
 \frac{\fa(\bsy;v,w)}{\| v \|_{\bcX} \|w\|_{\bcY}}
 \geq \mu
 \;,\quad
 \inf_{0\ne w \in \bcY} \sup_{0\ne v \in \bcX}
 \frac{\fa(\bsy;v,w)}{\| v \|_{\bcX} \|w\|_{\bcY}}
 \geq \mu
 \;.
\end{equation}
In particular, for every $f \in \bcY'$ and for every $\bsy \in \cS$,
the parametric operator equation
\begin{equation} \label{eq:parmOpEq}
 \mbox{find} \quad u(\bsy) \in \bcX:\quad
 \fa(\bsy;u(\bsy), w) \,=\,  {_{\bcY}}\langle w , f \rangle_{\bcY'}
 \quad
 \forall w \in \bcY
\end{equation}
admits a unique solution $u(\bsy)$ which satisfies the a-priori estimate
\begin{equation}\label{eq:apriori}
 \| u(\bsy) \|_{\bcX}
 \,\leq\,
 \frac{1}{\mu}\, \| f \|_{\bcY'}
\;.
\end{equation}
\end{theorem}

For a proof of the theorem, we refer to \cite[Theorem 2]{ScMCQMC12}.
\subsection{Parametric regularity of solutions}
\label{ssec:anadepsol}
In this subsection we study the dependence of the solution $u(\bsy)$ of
the parametric, variational problem \eqref{eq:parmOpEq} on the parameter
vector $\bsy$. 
In the following, let $\N_0^\N$ denote the set of sequences $\bsnu =
(\nu_j)_{j\geq 1}$ of nonnegative integers $\nu_j$, and let $|\bsnu| :=
\sum_{j\geq 1} \nu_j$. For $|\bsnu|<\infty$, we denote the partial
derivative of order $\bsnu$ of $u(\bsy)$ with respect to $\bsy$ by $\partial^\bsnu_\bsy u \,:=\,
(\partial^{|\bsnu|}u)/(\partial^{\nu_1}_{y_1}\partial^{\nu_2}_{y_2}\cdots)$.

\begin{theorem}\label{thm:Dsibound} \cite{CDS1,KunothCS2011}
Under Assumption~\ref{ass:AssBj}, there exists a constant $C_0 > 0$ such
that for every $f\in \bcY'$ and for every $\bsy\in \cS$, the partial
derivatives of the parametric solution $u(\bsy)$ of the parametric
operator equation \eqref{eq:main} with affine operator \eqref{eq:Baffine}
satisfy the bounds
\begin{equation} \label{eq:Dsibound}
\|(\partial^\bsnu_\bsy u)(\bsy)\|_\bcX
\,\le\,
C_0\, |\bsnu|! \,\bsbeta^\bsnu \| f\|_{\bcY'}
\quad \mbox{for all } \bsnu \in \N_0^\N \mbox{ with } |\bsnu|<\infty
\;,
\end{equation}
where $0! :=1$, $\bsbeta^\bsnu := \prod_{j\ge 1}
\beta_j^{\nu_j}$, with $\beta_j$ as in \eqref{eq:Bjsmall},
and $|\bsnu| = \sum_{j \ge 1} \nu_j$.
\end{theorem}
\subsection{Spatial regularity of solutions}
\label{sec:SpatReg}
We assume given {\em scales of smoothness spaces}
$\{ \cX_t \}_{t \geq 0}$, $\{ \cY_t \}_{t\geq 0}$,
with
\begin{equation}\label{eq:SmScal}
\begin{aligned}
 \cX &= \cX_0 \supset \cX_1 \supset \cX_2 \supset \cdots\;,
 &\cY &= \cY_0 \supset \cY_1 \supset \cY_2 \supset \cdots\;,
 \quad\mbox{and}
 \\
 \cX' &= \cX'_0 \supset \cX'_1 \supset \cX'_2 \supset \cdots\;,
 &\cY' &= \cY'_0 \supset \cY'_1 \supset \cY'_2 \supset \cdots
 \;.
\end{aligned}
\end{equation}
The scales are assumed to be defined also for noninteger values of the smoothness
parameter $t\geq 0$ by interpolation. For self-adjoint operators, usually
$\cX_t  = \cY_t$. For example, in diffusion problems in {\em convex
domains} $D$ considered in \cite{CDS1,KSS12}, the smoothness scales
\eqref{eq:SmScal} are $\cX = \cY = H^1_0(D)$, $\cX_1 = \cY_1 = (H^2\cap
H^1_0)(D)$, $\cY' = H^{-1}(D)$, $\cY'_1 = L^2(D)$. In a nonconvex polygon
(or polyhedron), analogous smoothness scales are available, but involve
Sobolev spaces with weights\footnote{Not to be confused with the weighted
Sobolev spaces in QMC error analysis, see, e.g., \cite{SW98,KSS11}.}.

In \cite{NS12}, this kind of abstract regularity result was established
for a wide range of second order parametric, elliptic systems in 2D and
3D, also for higher order regularity. Importantly, the smoothness scales
are then weighted Sobolev spaces $\cK^{t+1}_{a+1}(D)$ of Kondratiev type
in $D$, and hence $\cX_t = \cK^{t+1}_{a+1}(D)$, $\cY'_t =
\cK^{t-1}_{a-1}(D)$ in this case. The Finite Element spaces which realize
the maximal convergence rates (beyond order one) are regular, simplicial
families in the sense of Ciarlet, on suitably refined meshes which
compensate for the corner and edge singularities.

In the ensuing convergence analysis of Petrov-Galerkin
discretizations of \eqref{eq:main}, we will assume that the data
regularity $f\in\cY'_t$ for some $t>0$ implies that
\begin{equation}\label{eq:Regul}
\forall\bsy\in \cS: \quad u(\bsy) = A(\bsy)^{-1} f \in \cX_t \;.
\end{equation}
Parametric regularity is available for numerous parametric differential
equations (see \cite{SchwabGittelsonActNum11,HaSc11,HoaSc12Wave,CCS2,KunothCS2011}
and the references there) as well as for posterior densities in Bayesian inverse
problems \cite{SS12,SS13}.
\subsection{Petrov-Galerkin discretization}
\label{ssec:GalDisc}

Let $\{ \cX^h \}_{h>0}\subset\cX$ and $\{ \cY^h \}_{h>0}\subset\cY$ be two
families of finite dimensional subspaces which are dense in $\cX$ and in
$\cY$, respectively. We will also assume the {\em approximation
properties}: for $0<t \leq \bar{t}$ and $0 < t' \leq \bar{t'}$, and for
$0< h \leq h_0$, there hold
\begin{equation} \label{eq:apprprop}
\begin{aligned}
\forall v \in \cX_t\;
&:
\quad
\inf_{v^h\in \cX^h} \| v - v^h \|_{\cX}
\,\leq\,
C_t\, h^t\, \| v \|_{\cX_t} \;,
\\
\forall w \in \cY_{t'}\;
&:
\quad
\inf_{w^h\in \cY^h} \| w - w^h \|_{\cY}
\,\leq\,
C_{t'}\, h^{t'}\, \| w \|_{\cY_{t'}}
\;.
\end{aligned}
\end{equation}
The maximum amount of smoothness in the scale $\cX_t$, denoted by
$\bar{t}$, depends on the problem class under consideration and on the
Sobolev scale: e.g., for elliptic problems in polygonal domains, it is
well known that choosing for $\cX_t$ the usual Sobolev spaces will allow
\eqref{eq:Regul} with $t$ only in a possibly small interval
$0< t \leq \bar{t}$, whereas choosing $\cX_t$ as Sobolev spaces with weights
will allow rather large values of $\bar{t}$ (see, e.g., \cite{NS12}).
Corresponding to \eqref{eq:Regul}, we shall assume that
\begin{equation} \label{eq:assW1infty} 
  \forall\,0\leq t \leq \bar{t} :\quad
  \sup_{\bsy\in \cS} \| A(\bsy)^{-1}  \|_{\cL(\cY'_t, \cX_t)}  < \infty\;.
\end{equation}
\begin{theorem}\label{prop:stab}
Assuming that the subspace sequences $\{ \cX^h \}_{h>0}\subset\cX$
and $\{ \cY^h \}_{h>0}\subset\cY$ are stable, i.e., there exist $\bar{\mu}
> 0$ and $h_0 > 0$ such that for every $0<h \leq h_0$, there hold the
uniform (with respect to $\bsy\in \cS$) discrete inf-sup conditions
\begin{align}\label{eq:Bhinfsup1}
&\forall \bsy \in \cS:
\quad
\inf_{0\ne v^h \in \bcX^h} \sup_{0\ne w^h \in \bcY^h}
\frac{\fa(\bsy;v^h,w^h)}{\| v^h \|_{\bcX} \|w^h\|_{\bcY}}
\geq \bar{\mu} > 0\;,
\\
\label{eq:Bhinfsup2}
&\forall \bsy\in \cS:\quad
\inf_{0\ne w^h \in \bcY^h} \sup_{0\ne v^h \in \bcX^h}
\frac{\fa(\bsy;v^h,w^h)}{\| v^h \|_{\bcX} \|w^h\|_{\bcY}}
\geq \bar{\mu}>0
\;.
\end{align}
Then, for every $0<h \leq h_0$ and for every $\bsy \in \cS$, the
Petrov-Galerkin approximations $u^h(\bsy)\in\cX^h$, given by
\begin{equation} \label{eq:parmOpEqh}
\mbox{find} \; u^h(\bsy) \in \bcX^h :
\quad
\fa(\bsy;u^h(\bsy),w^h) =
{_{\cY}}\langle w^h, f \rangle_{\cY'}
\quad
\forall w^h\in \bcY^h\;,
\end{equation}
admits a unique solution $u^h(\bsy)$ which satisfies the a-priori estimate
\begin{equation}\label{eq:FEstab}
 \| u^h(\bsy) \|_{\bcX} \,\le\, \frac{1}{\bar{\mu}}\, \| f \|_{\bcY'}
\;.
\end{equation}
Moreover, there exists a constant $C>0$ such that for all $\bsy\in \cS$ quasioptimality holds,
\begin{equation} \label{eq:quasiopt}
 \| u(\bsy) - u^h(\bsy) \|_{\cX}
 \,\le\, \frac{C}{\bar{\mu}} \inf_{0\ne v^h\in \cX^h} \| u(\bsy) - v^h\|_{\cX}
\;.
\end{equation}
\end{theorem}

We remark that under Assumption \ref{ass:AssBj}, the validity of the
discrete inf-sup conditions for the nominal bilinear form
$\fa_0(\cdot,\cdot)$, see~\eqref{eq:B0infsup},
with constant
$\bar{\mu}_0>0$ independent of $h$, implies \eqref{eq:Bhinfsup1} and
\eqref{eq:Bhinfsup2} for the bilinear form $\fa(\bsy;\cdot,\cdot)$
with $\bar{\mu} = (1-\kappa/2) \bar{\mu}_0>0$.

\begin{theorem}\label{thm:FEconvrate}
Under Assumption~\ref{ass:AssBj} and condition \eqref{eq:assW1infty}, for
every $f\in \cY'$ and for every $\bsy\in \cS$, the approximations
$u^h(\bsy)$ are stable, i.e., \eqref{eq:FEstab} holds. For every $f\in
\cY'_t$ with $0<t\le \bar{t}$, there exists a constant $C>0$ such that as
$h\rightarrow 0$ there holds
\begin{equation} \label{eq:FEconvu}
  \| u(\bsy) - u^h(\bsy) \|_\cX \,\le\, C\, h^t\, \| f \|_{\cY'_t }
  \;.
\end{equation}
\end{theorem}

Since we are interested in the expectations of functionals of the
parametric solution, see \eqref{eq:int},
we will also impose a regularity
assumption on the functional $G(\cdot)\in \cX'$:
\begin{equation}\label{eq:regG}
\exists\; 0< t'\leq \bar{t}: \quad
G(\cdot) \in \cX'_{t'}
\;,
\end{equation}
and the {\em adjoint regularity}:
for $t'$ as in \eqref{eq:regG} there exists $C_{t'}>0$
such that for every $\bsy\in \cS$,
\begin{equation}\label{eq:refAd}
w(\bsy) = (A^*(\bsy))^{-1} G \in \cY_{t'} \;,
\quad
 \| w(\bsy) \|_{\cY_{t'}} \,\le\, C_{t'} \, \| G \|_{\cX'_{t'}} \;.
\end{equation}
Moreover, we see from \eqref{eq:qmcG} that
the discretization error
of $G(u(\bsy))$ is of interest as well. It is known that
$|G(u(\bsy))-G(u^h(\bsy))|$
may converge faster than $\|u(\bsy)-u^h(\bsy)\|_\cX$.
\begin{theorem} \label{thm:FEGconv}
Under Assumption~\ref{ass:AssBj} and the conditions \eqref{eq:assW1infty}
and \eqref{eq:refAd}, for every $f\in \cY'_{t}$ with $0<t\leq \bar{t}$,
for every $G(\cdot)\in \cX'_{t'}$ with $0<t'\leq \bar{t}$ and for every
$\bsy\in \cS$, as $h\to 0$, there exists a constant $C>0$ independent of
$h>0$ and of $\bsy\in U$ such that the Petrov-Galerkin approximations
$G(u^h(\bsy))$ satisfy
\begin{align} \label{eq:Gconvest}
  \left| G(u(\bsy)) - G(u^h(\bsy)) \right|
  &\,\le\, C\, h^{{\tau}} \, \| f \|_{\cY'_{t}} \, \| G \|_{\cX'_{t'}}
\;.
\end{align}
where $0<\tau := t+t'$.
\end{theorem}

The result follows from a (classical)
Aubin-Nitsche duality argument \cite{NitSch74}.

\subsection{Dimension truncation}
\label{sec:dimtrunc}
In order to approximate the integral \eqref{eq:int} by QMC methods, we
truncate the infinite sum in \eqref{eq:Baffine} to $s$ terms, as indicated
in \eqref{eq:qmcG}. We denote by $u_s(\bsy)$ the solution of the corresponding parametric weak
problem \eqref{eq:parmOpEq}. Then Theorem~\ref{thm:BsigmaInv} holds when
$u(\bsy)$ is replaced by $u_s(\bsy)$. In addition to the assumption \eqref{eq:psumpsi0}, which implies
$\sum_{j\ge 1} \beta_j^p < \infty$ with $\beta_j$ defined as in
\eqref{eq:Bjsmall}, we assume that the operators $A_j$ are enumerated so
that
\begin{equation} \label{eq:ordered} 
  \beta_1 \ge \beta_2 \ge \cdots \ge \beta_j \ge \, \cdots\;.
\end{equation}

\begin{theorem} \label{thm:trunc}
Under Assumption~\ref{ass:AssBj}, for every $f\in \cY'$, for every
$\bsy\in U$ and for every $s\in\bbN$, the solution $u_s(\bsy)$ of the
$s$-term truncated parametric weak problem \eqref{eq:parmOpEq} satisfies,
with $\beta_j$ as defined in \eqref{eq:Bjsmall},
\begin{equation}\label{eq:Vdimtrunc}
  \| u(\bsy) - u_s(\bsy) \|_\cX
  \,\le\, \frac{C}{\mu}\, \|f\|_{\cY'}\,
  \sum_{j\ge s+1} \beta_j
\end{equation}
for some constant $C>0$ independent of $f$.
Moreover, for every $G(\cdot)\in \cX'$,
we have
\begin{equation}\label{eq:Idimtrunc}
  |I(G(u))- I(G(u^s))|
  \,\le\, \frac{\tilde{C}}{\mu}\, \|f\|_{\cY'}\, \|G\|_{\cX'}\,
  \bigg(\sum_{j\ge s+1} \beta_j\bigg)^2
\end{equation}
for some constant $\tilde{C}>0$ independent of $f$ and $G$.
In addition,
if conditions~\eqref{eq:psumpsi0} and \eqref{eq:ordered}
hold, then
\[
  \sum_{j\ge s+1} \beta_j
  \,\le\,
  \min\left(\frac{1}{1/p-1},1\right)
  \bigg(\sum_{j\ge1} \beta_j^p \bigg)^{1/p}
  s^{-(1/p-1)}\;.
\]
\end{theorem}

The proof is a generalization of \cite[Theorem~5.1]{KSS12}.

\section{Higher order QMC error analysis}
\label{Sc:HiOrdQMC}

Throughout this section, we consider a \emph{general} $s$-variate
integrand $F$ defined over the unit cube $[0,1]^s$, and we approximate the
$s$-dimensional integral
\begin{equation}\label{eq:IsF}
 I_s(F) \,:=\,
 \int_{[0,1]^s} F(\bsy) \,\rd\bsy
\end{equation}
by an $N$-point QMC method, i.e., an equal-weight quadrature rule of the
form
\begin{equation}\label{eq:QNs}
  Q_{N,s}(F) \,:=\,
  \frac{1}{N} \sum_{n=0}^{N-1} F(\bsy_n)\;,
\end{equation}
with judiciously chosen points $\bsy_0,\ldots,\bsy_{N-1} \in [0,1]^s$. We
shall always bear in mind the special case where the integrand
$F(\bsy) = G(u^h_s(\bsy-\bhalf))$ is a linear functional applied to the
solution of a parametric operator equation, see \eqref{eq:int} and
\eqref{eq:qmcG}. 
\begin{theorem}[Main Result] \label{thm:main}
Let $s\ge 1$ and $N = b^m$ for $m\ge 1$ and prime $b$. Let $\bsbeta =
(\beta_j)_{j\ge 1}$ be a sequence of positive numbers, let $\bsbeta_s =
(\beta_j)_{1\le j \le s}$, and assume that
\begin{equation} \label{p-sum}
  \exists\, 0<p\le 1 : \quad \sum_{j=1}^\infty \beta_j^p < \infty\;.
\end{equation}
Define
\begin{equation} \label{alpha}
  \alpha \,:=\, \lfloor 1/p \rfloor +1 \;.
\end{equation}
If \eqref{p-sum} holds only with $p=1$, we assume additionally
that $\sum_{j=1}^\infty \beta_j$ is
small as in \eqref{eq:small} below.
Suppose we have an integrand $F$ whose partial derivatives satisfy
\begin{equation} \label{eq:like-norm}
 \forall\, \bsnu \in \{0, 1, \ldots, \alpha\}^s: \quad
 | (\partial^{\bsnu}_\bsy F)(\bsy)| \,\le\, c\, |\bsnu|!\, \bsbeta_s^{\bsnu}
\end{equation}
for some constant $c>0$.
Then, an interlaced polynomial lattice rule of
order $\alpha$ with $N$ points can be constructed using a fast
component-by-component algorithm, with cost
$\calO(\alpha\,s\, N\log N + \alpha^2\,s^2 N)$ operations,
such that
\[
  |I_s(F) - Q_{N,s}(F)|
  \,\le\, C_{\alpha,\bsbeta,b,p}\, N^{-1/p} \;,
\]
where $C_{\alpha,\bsbeta,b,p} < \infty$ is a constant independent of $s$
and $N$.
\end{theorem}
\begin{theorem}
If \eqref{eq:like-norm} in Theorem~\ref{thm:main} is replaced by
\begin{equation} \label{eq:like-norm2}
  \forall\, \bsnu \in \{0, 1, \ldots, \alpha\}^s: \quad
 \sup_{\bsy\in U} | (\partial^{\bsnu}_\bsy F)(\bsy)|
\,\le\,
c\, \bsnu!\, \bsbeta_s^{\bsnu}
\;,
\end{equation}
then the result still holds, but the cost of the CBC algorithm is only
$\calO(\alpha\,s\, N\log N)$ operations, and
the additional condition \eqref{eq:small} is not required when $p=1$.
\end{theorem}

We shall prove these theorems in stages in the following subsections,
starting by introducing a new function space setting motivated by
\eqref{eq:like-norm} and \eqref{eq:like-norm2}.
Before we proceed, we note that the ``interlaced polynomial lattice
rules'' used in the theorems above, to be formally introduced in
\S\ref{sec:PolLattDims}, are \emph{deterministic}. We always have $\alpha
\ge 2$, where $\alpha=2$ is obtained with $p=1$. This indicates that we
require interlaced polynomial lattice rules of order $2$ to achieve a
convergence rate of $N^{-1}$, with the implied constant independent of the
dimension.

We also remark that \eqref{eq:Dsibound} in Theorem~\ref{thm:Dsibound}
yields an integrand $F(\bsy) = G(u_s(\bsy))$ (after dimension truncation)
which satisfies \eqref{eq:like-norm}, with $c=C_0\| f
\|_{\cY'}\,\|G\|_{\cX'}$.

\subsection{A new function space setting for smooth integrands}
\label{sec:space}

In this subsection we consider numerical integration for \eqref{eq:IsF}
for smooth integrands $F$ of $s$ variables using a family of QMC rules
called \emph{digital nets}, see, e.g., \cite{Nie92,DiPi10}. We define in
the following a class of function spaces on the unit cube in finite
dimension $s$ which contain the integrands $F(\bsy) =
G(u^h_s(\bsy-\bhalf))$, that is, integrands which arise from linear
functionals of solutions of the parametric operator equation
\eqref{eq:main}.
\begin{definition}[Norm and function space] \label{def_F_norm}
Let $\alpha, s\in\bbN$, $1\le q \le \infty$ and $1\le r \le \infty$, and let $\bsgamma =
(\gamma_\setu)_{\setu\subset\bbN}$ be a collection of nonnegative real
numbers, known as \emph{weights}. Assume further that $F: [0,1]^s \to
\mathbb{R}$ has partial derivatives of orders up to $\alpha$ with respect
to each variable. Then we define the norm of $F$ by a higher order unanchored Sobolev norm
%
\begin{align}\label{eq:defFabs}
 \|F\|_{s,\alpha,\bsgamma,q,r}
 &:=
 \Bigg( \sum_{\setu\subseteq\{1:s\}} \Bigg( \gamma_\setu^{-q}
 \sum_{\setv\subseteq\setu} \sum_{\bstau_{\setu\setminus\setv} \in \{1:\alpha\}^{|\setu\setminus\setv|}} \\
 &\qquad\qquad\quad
 \int_{[0,1]^{|\setv|}} \bigg|\int_{[0,1]^{s-|\setv|}} \!
 (\partial^{(\bsalpha_\setv,\bstau_{\setu\setminus\setv},\bszero)}_\bsy F)(\bsy) \,\rd \bsy_{\{1:s\} \setminus\setv}
 \bigg|^q \rd \bsy_\setv \Bigg)^{r/q} \Bigg)^{1/r}, \nonumber
\end{align}
with the obvious modifications if $q$ or $r$ is infinite. Here $\{1:s\}$
is a shorthand notation for the set $\{1,2,\ldots,s\}$, and
$(\bsalpha_\setv,\bstau_{\setu\setminus\setv},\bszero)$ denotes a sequence
$\bsnu$ with $\nu_j = \alpha$ for $j\in\setv$, $\nu_j = \tau_j$ for
$j\in\setu\setminus\setv$, and $\nu_j = 0$ for $j\notin\setu$.
Let $\calW_{s,\alpha,\bsgamma,q,r}$ denote the Banach space of all such
functions $F$ with finite norm.
\end{definition}

\begin{definition}[Digital net]
Let $b$ be prime and $\alpha,s,m\in\bbN$. Let $C_1,\ldots,C_s$ be $\alpha
m\times m$ matrices over $\bbZ_b$; these are known as the \emph{generating
matrices}. For each integer $0 \le n < b^m$, let $n = \eta_0 + \eta_1 b +
\cdots + \eta_{m-1} b^{m-1}$ be the $b$-adic expansion of $n$. For each
$1\le j\le s$ we compute
$ 
  (\zeta_1,\zeta_2,\ldots,\zeta_{\alpha m})^\top \,=\, C_j\, (\eta_0, \eta_1,\ldots,\eta_{m-1})^\top\;,
$ 
set
$ 
  y_j^{(n)} \,=\, \frac{\zeta_1}{b} + \frac{\zeta_2}{b^2} + \cdots + \frac{\zeta_{\alpha m}}{b^{\alpha m}}
$ 
and set
$ \bsy_n = (y_1^{(n)}, y_2^{(n)}, \ldots, y_s^{(n)})$.

Then, the resulting point set $\calS = \{\bsy_n\}_{n=0}^{b^m-1}
\subset [0,1]^s$ is called a \emph{digital net}.
\end{definition}

We derive an upper bound on the \emph{worst case error} of a digital net in
$\calW_{s,\alpha,\bsgamma,q,r}$.

\begin{theorem}[Worst case error bound] \label{thm:wce}
Let $\alpha, s\in\bbN$ with $\alpha>1$, $1\le q\le \infty$ and $1\le r\le \infty$, and let
$\bsgamma = (\gamma_\setu)_{\setu\subset\bbN}$ denote a collection of
weights. Let $r'\ge 1$ satisfy $1/r + 1/r' = 1$. Let $b$ be
prime, $m\in\bbN$, and let $\calS=\{\bsy_n\}_{n=0}^{b^m-1}$ denote a
digital net with generating matrices $C_1,\ldots,C_s\in\bbZ_b^{\alpha
m\times m}$. Then we have
\[
  \sup_{\|F\|_{s,\alpha,\bsgamma,q,r} \le 1}
  \left| \frac{1}{b^m} \sum_{n=0}^{b^m-1} F(\bsy_n) - \int_{[0,1]^s} F(\bsy) \,\mathrm{d} \bsy \right|
  \,\le\, e_{s,\alpha,\bsgamma,r'}(\calS)\;,
\]
with
\begin{align} \label{def-B}
  e_{s,\alpha,\bsgamma,r'}(\calS)
  \,:=\, \Bigg(\sum_{\emptyset \neq \setu \subseteq \{1:s\}}
  \bigg(C_{\alpha,b}^{|\setu|}\, \gamma_\setu \sum_{\bsk_\setu \in \setD_\setu^*}
  b^{-\mu_{\alpha}(\bsk_\setu)} \bigg)^{r'} \Bigg)^{1/r'}\;.
\end{align}
Here $\setD_\setu^*$ is the ``dual net without $0$ components'' projected
to the components in $\setu$, defined by
\begin{equation} \label{dual}
  \setD_\setu^* \,:=\, \left\{ \bsk_\setu \in \bbN^{|\setu|}\,:\,
  \sum_{j\in\setu} C_j^\top {\rm tr}_{\alpha m}(k_j) = \bszero \in \bbZ_b^m \right\} \;,
\end{equation}
where ${\rm tr}_{\alpha m}(k) := (\varkappa_0, \varkappa_1, \ldots,
\varkappa_{\alpha m-1})^\top$ if $k = \varkappa_0 + \varkappa_1 b
+\varkappa_2 b^2 + \cdots$ with $\varkappa_i\in \{0,\ldots,b-1\}$.
Moreover, we have $\mu_{\alpha}(\bsk_\setu) = \sum_{j\in\setu}
\mu_\alpha(k_j)$ with
\begin{equation} \label{def-mu-k}
  \mu_\alpha(k)
  \,:=\,
  \begin{cases}
  0 & \mbox{if } k = 0, \\
  a_1 + \cdots + a_{\min(\alpha,\rho)} & 
  \begin{aligned}
   \mbox{if }
  k &= \kappa_1 b^{a_1-1} + \cdots + \kappa_\rho b^{a_\rho-1} \mbox{ with} \\
    &\kappa_i\in \{1,\ldots,b-1\} \mbox{ and } a_1>\cdots>a_\rho>0,
  \end{aligned}
  \end{cases}
\end{equation}
and
\begin{align}\label{eq:Cab}
 &C_{\alpha,b}
 \,:=\,
 \max\left(\frac{2}{(2\sin\frac{\pi}{b})^{\alpha}},\max_{1\le z\le\alpha-1}\frac{1}{(2\sin\frac{\pi}{b})^z}\right)
 \nonumber\\
 &\qquad\qquad\qquad\times
 \left(1+\frac{1}{b}+\frac{1}{b(b+1)}\right)^{\alpha-2}
 \left(3 + \frac{2}{b} + \frac{2b+1}{b-1} \right)\;.
\end{align}
\end{theorem}

\begin{proof}
Assume that $\|F\|_{s,\alpha,\bsgamma,q,r} < \infty$. Let $B_\tau$ denote
the Bernoulli polynomial of degree~$\tau$ and let $b_\tau = (\tau!)^{-1}
B_\tau$. Furthermore, let $\widetilde{b}_\tau$ denote the one-periodic
extension of the polynomial $b_\tau:[0,1)\to\mathbb{R}$. We claim that $F$
can be represented by
\begin{align}\label{eq_anova}
  F(\bsy) \,=\, \sum_{\setu\subseteq\{1:s\}} F_\setu(\bsy_\setu)\;,
\end{align}
where
\begin{align}\label{eq_repre}
  F_\setu(\bsy_\setu)
  &\,=\,
  \sum_{\setv\subseteq\setu} \sum_{\bstau_{\setu\setminus\setv} \in \{1:\alpha\}^{|\setu\setminus\setv|}}
  \bigg(\prod_{j\in\setu\setminus\setv} b_{\tau_j}(y_j)\bigg) (-1)^{(\alpha+1) |\setv|} \nonumber\\
  &\qquad\qquad\qquad\qquad
  \times\int_{[0,1]^s} (\partial^{(\bstau_{\setu\setminus\setv}, \bsalpha_\setv,\bszero)}_\bsx F)(\bsx)
  \prod_{j\in\setv} \widetilde{b}_{\alpha}(x_j-y_j) \,\rd \bsx\;. 
\end{align}
To see this, consider the set $\mathcal{P}$ of all polynomials defined on
$[0,1]^s$. For these functions \eqref{eq_anova} and \eqref{eq_repre} hold.
The set $\mathcal{P}$ is dense in $\calW_{s,\alpha,\bsgamma,q,r}$. Let $F
\in \calW_{s,\alpha,\bsgamma,q,r}$. Since $\mathcal{P}$ is dense, there
exists a sequence of functions $(F_n)_{n\ge 1}$ in $\mathcal{P}$ such that
$\|F_n- F\|_{s,\alpha,\bsgamma,q,r} \to 0$ as $n \to \infty$. Since the
set $\calW_{s,\alpha,\bsgamma,q,r}$ is complete with respect to the norm
$\|\cdot\|_{s,\alpha,\bsgamma,q,r}$, $(F_n)$ is a Cauchy sequence. It
follows that $(\partial^{(\bstau_{\setu\setminus\setv},
\bsalpha_\setv,\bszero)}_\bsx F_n)(\bsx)$ is a Cauchy sequence in
$L_q$. Let the limit of these sequences be denoted by
$(\partial^{(\bstau_{\setu\setminus\setv}, \bsalpha_\setv,\bszero)}_\bsx
\widetilde{F})(\bsx)$ and let $\widetilde{F}$ be defined via
\eqref{eq_anova} and \eqref{eq_repre}. Then $\|F -
\widetilde{F}\|_{s,\alpha,\bsgamma,q,r} \le \|F -
F_n\|_{s,\alpha,\bsgamma,q,r} + \| F_n - \widetilde{F}
\|_{s,\alpha,\bsgamma,q,r}$. Then $\| F - F_n\|_{s,\alpha,\bsgamma,q,r}\to 0$ by the definition of
$F_n$ and $\| F_n- \widetilde{F} \|_{s,\alpha,\bsgamma,q,r}\to 0$ by the
definition of $(\partial^{(\bstau_{\setu\setminus\setv},
\bsalpha_\setv,\bszero)}_\bsx \widetilde{F})(\bsx)$. Therefore the claim
is shown.

Note that \eqref{eq_anova} and \eqref{eq_repre} together is the
\emph{ANOVA decomposition} of $F$, since for any nonempty $\setu$ we have
$\int_0^1 F_\setu(\bsy_\setu)\,\rd y_j = 0$ whenever $j\in\setu$. This
follows from the property that $\int_0^1 b_\tau(y) \,\rd y = 0$ for all
$\tau \ge 1$. Moreover, we have
\begin{align*}
  &\|F_\setu\|_{s,\alpha,\bsgamma,q,r} \\
  &\,=\, \gamma_\setu^{-1}
  \Bigg(\sum_{\setv\subseteq\setu} \sum_{\bstau_{\setu\setminus\setv} \in \{1:\alpha\}^{|\setu\setminus\setv|}}
  \int_{[0,1]^{|\setv|}} \bigg|\int_{[0,1]^{s-|\setv|}}\!\!
  (\partial^{(\bstau_{\setu\setminus\setv}, \bsalpha_\setv,\bszero)}_\bsy F_\setu)(\bsy)
  \,\rd \bsy_{\{1:s\}\setminus\setv} \bigg|^q
  \rd \bsy_\setv\Bigg)^{1/q}.
\end{align*}
Thus we have the norm decomposition
$ 
  \|F\|_{s,\alpha,\bsgamma,q,r}
  \,=\, (
  \sum_{\setu \subseteq \{1:s\}} \|F_\setu\|_{s,\alpha,\bsgamma,q,r}^r 
  )^{1/r}.
$ 

Let $\widehat{F}(\bsk)$ denote the $\bsk$th Walsh coefficient of $F$ and
$\widehat{F_\setu}(\bsk_\setu)$ denote the $\bsk_\setu$th Walsh
coefficient of $F_\setu$. (We refer to \cite{D09, DiPi10} and the
references there for more information on Walsh function expansions.) Then
\begin{equation*}
  F(\bsy) \,=\, \sum_{\bsk \in \mathbb{N}_0^s} \widehat{F}(\bsk) \wal_{\bsk}(\bsy)
  \,=\, \sum_{\setu\subseteq \{1:s\}}
  \sum_{\bsk_\setu \in \mathbb{N}^{|\setu|}} \widehat{F_\setu}(\bsk_\setu) \wal_{\bsk_\setu}(\bsy_\setu)
  \,=\, \sum_{\setu \subseteq \{1:s\}} F_\setu(\bsy_\setu)\;.
\end{equation*}
Note that $\int_0^1 b_\tau(y) \,\rd y = 0$ for all $\tau \ge 1$ and
$\int_0^1 \wal_k(y) \,\rd y = 0$ for all $k \ge 1$. Thus
\begin{align*}
  \widehat{F}(\bsk_\setu, \boldsymbol{0})
  \,=\, \int_{[0,1]^s} F(\bsy)\, \overline{\wal_{\bsk_\setu}(\bsy_\setu) } \,\rd \bsy_\setu
  \,=\, \int_{[0,1]^{|\setu|}} F_\setu(\bsy_\setu)\, \overline{\wal_{\bsk_\setu}(\bsy_\setu) } \,\rd \bsy_\setu
  \,=\, \widehat{F_\setu}(\bsk_\setu)\;.
\end{align*}
{}From the 
character property of digital nets, see, e.g.,
\cite[Lemma~4.75]{DiPi10}, we obtain
\begin{align} \label{eq_error_bound}
  \left| \frac{1}{b^m} \sum_{n=0}^{b^m-1} F(\bsy_n) - \int_{[0,1]^s} F(\bsy) \,\mathrm{d} \bsy \right|
  \,\le\,
  \sum_{\emptyset\neq \setu \subseteq\{1:s\}} \sum_{\bsk_\setu \in \setD_\setu^*}
 |\widehat{F}(\bsk_\setu, \boldsymbol{0})|\;.
\end{align}

We now explain how to obtain a bound on the Walsh coefficients
$\widehat{F}(\bsk_\setu, \boldsymbol{0})$. For $s=1$, \cite[Theorem~14]{D09} states that for $F \in
\calW_{1,\alpha,\bsgamma,q,r}$ and $k\in\bbN$ we have
\begin{align}
  |\widehat{F}(k)|
  &\,\le\, \sum_{z=\rho}^\alpha \left|\int_0^1 F^{(z)}(x) \,\rd x \right| \; \frac{b^{-\mu_{z,{\rm per}}(k)}}{(2\sin\frac{\pi}{b})^z} \;
  \left(1 + \frac{1}{b} + \frac{1}{b(b+1)}\right)^{\max(0,z-2)} \label{equ:sumz}\\
  &\quad + \int_0^1 |F^{(\alpha)}(x)| \,\rd x \; \frac{2  b^{-\mu_{\alpha, {\rm per}}(k)}}{(2\sin\frac{\pi}{b})^\alpha}
  \left(1 + \frac{1}{b} + \frac{1}{b(b+1)}\right)^{\alpha - 2}
  \left(3 + \frac{2}{b} + \frac{2b+1}{b-1}\right)\; \nonumber,
\end{align}
where $\rho = 0$ if $k=0$, otherwise $\rho$ is given by the expansion of
$k$ in \eqref{def-mu-k}, and where
\[
 \mu_{z, {\rm per}}(k)
 \,:=\,
 \begin{cases}
 0 & \mbox{for } z = 0 \mbox{ and } k \ge 0, \\
 0 & \mbox{for } k =  0 \mbox{ and } z \ge 0, \\
 a_1 + \cdots + a_\rho + (z- \rho) a_\rho & \mbox{for } 1 \le \rho < z, \\
 a_1 + \cdots + a_{\rho} & \mbox{for } \rho \ge z.
 \end{cases}
\]
Moreover, for $\rho > \alpha$ the empty sum \eqref{equ:sumz} is defined as $0$.
Note that for $k \in \mathbb{N}$ with $\rho$ nonzero
digits in its base $b$ expansion, we have for $z \ge \rho$ that
$\mu_{z, {\rm per}}(k) \ge \mu_{\alpha}(k)$ and therefore
$b^{-\mu_{z, {\rm per}}(k) } \le b^{-\mu_{\alpha}(k)}$. Some further estimates yield
\[
  |\widehat{F}(k)|
  \,\le\, C_{\alpha, b}\, b^{-\mu_\alpha(k)} \gamma_{\{1\}}\,
  \|F_{\{1\}} \|_{1,\alpha,\bsgamma,q,r}\;.
\]
As mentioned in \cite[Remark~15]{D09}, this bound can be extended to $s >
1$. In this case one uses the representation \eqref{eq_repre}. Since the
proof of \cite[Theorem~14]{D09} uses bounds on the Walsh coefficients of
the Bernoulli polynomials, which appear in \eqref{eq_repre} in product
form, we obtain a bound of the following product form:
\begin{equation*}
  |\widehat{F}(\bsk_\setu,\boldsymbol{0})|
  \,\le\, C_{\alpha, b}^{|\setu|} \, b^{-\mu_\alpha(\bsk_\setu)} \gamma_\setu\,
  \|F_\setu \|_{s,\alpha,\bsgamma,q,r}\;.
\end{equation*}
Substituting this into \eqref{eq_error_bound} gives
\begin{align*}
  &\left| \frac{1}{b^m} \sum_{n=0}^{b^m-1} F(\bsy_n) - \int_{[0,1]^s} F(\bsy) \,\rd \bsy \right|
  \,\le\,
  \sum_{\emptyset\neq \setu \subseteq\{1:s\}}
  \|F_\setu\|_{s,\alpha,\bsgamma,q,r}\,C_{\alpha,b}^{|\setu|}\, \gamma_\setu
  \sum_{\bsk_\setu \in \setD_\setu^*}b^{-\mu_\alpha(\bsk_\setu)} \\
  &\qquad\,\le\,
  \Bigg(\sum_{\setu \subseteq\{1:s\}}
  \|F_\setu\|_{s,\alpha,\bsgamma,q,r}^r\Bigg)^{1/r}
  \Bigg(\sum_{\emptyset\neq \setu \subseteq\{1:s\}}
  \bigg(C_{\alpha,b}^{|\setu|}\, \gamma_\setu
  \sum_{\bsk_\setu \in \setD_\setu^*}b^{-\mu_\alpha(\bsk_\setu)}\bigg)^{r'}
  \Bigg)^{1/r'} \;,
\end{align*}
which yields the worst case error bound in the theorem. \qquad\end{proof}
%
\begin{remark}
Theorem~\ref{thm:wce} also holds for any digitally shifted digital
net with any digital shift. To define the digital shift, consider $y, \sigma \in [0,1)$ with $b$-adic expansion
$y= y_1 b^{-1} + y_2 b^{-2} + \cdots$ and $\sigma = \sigma_1 b^{-1} +
\sigma_2 b^{-2} + \cdots$ with the assumption that infinitely many digits
are different from $b-1$. Then we set $y \oplus \sigma = z_1 b^{-1} + z_2
b^{-2} + \cdots$ where $z_i = y_i + \sigma_i \pmod{b}$. For vectors $\bsy,
\bssigma \in [0,1)^s$ we define $\bsy \oplus \bssigma$ component-wise. If
$\bsy_0, \bsy_1,\ldots, \bsy_{b^m-1}$ is a digital net, then we call
$\bsy_0 \oplus \bssigma, \bsy_1 \oplus \bssigma, \ldots, \bsy_{b^m-1}
\oplus \bssigma$ a digitally shifted digital net with digital shift
$\bssigma \in [0,1)^s$.

If we replace the points $\bsy_n$ for $0 \le n < b^m$ in
Theorem~\ref{thm:wce} by $\bsy_n \oplus \bssigma$ for $0 \le n < b^m$ for
some arbitrary vector $\bssigma \in [0,1)^s$, then the statement of the
theorem still holds. The only change in the proof is in
\eqref{eq_error_bound}.
There we have
\begin{align} \label{eq_error_bound2}
  &\left| \frac{1}{b^m} \sum_{n=0}^{b^m-1}
  F(\bsy_n \oplus \bssigma) - \int_{[0,1]^s} F(\bsy) \,\mathrm{d} \bsy \right|
  \,\le\,
  \left|  \sum_{\emptyset\neq \setu \subseteq\{1:s\}} \sum_{\bsk_\setu \in \setD_\setu^*}
  \widehat{F}(\bsk_\setu, \boldsymbol{0}) \wal_{(\bsk_\setu, \boldsymbol{0})}(\bssigma) \right|\; \nonumber
  \\
  &\,\le \, \sum_{\emptyset\neq \setu \subseteq\{1:s\}} \sum_{\bsk_\setu \in \setD_\setu^*}
  |\widehat{F}(\bsk_\setu, \boldsymbol{0})|\; |\wal_{(\bsk_\setu, \boldsymbol{0})}(\bssigma)| \;.
\end{align}
Since $|\wal_{(\bsk_\setu, \boldsymbol{0})}(\bssigma)| =1$,
\eqref{eq_error_bound2} coincides with the right-hand side of
\eqref{eq_error_bound}. The remaining part of the proof stays unchanged.
\end{remark}

Note that we require $\alpha > 1$ in Theorem \ref{thm:wce}
to ensure the
convergence of the expression $\sum_{\bsk_\setu \in\setD^*_\setu}
b^{-\mu_\alpha(\bsk_\setu)}$. This expression does not converge for
$\alpha = 1$.

The quantity $e_{s,\alpha,\bsgamma,r'}(\calS)$ can be used as an error
criterion to obtain good digital nets for our function space setting. It
is a generalization of the criterion used in \cite{BDGP11, BDLNP12}: there
the weights are of a product form $\gamma_\setu = \prod_{j\in\setu}
\gamma_j$ for some nonnegative sequence $(\gamma_j)_{j\ge 1}$, whereas
here the weights take a general form.

We stress that this quantity $e_{s,\alpha,\bsgamma,r'}(\calS)$ does not
depend on the parameter $q$. Furthermore, it is more convenient to work
with an upper bound which can be obtained by taking $r = \infty$ and
$r'=1$, i.e., $e_{s,\alpha,\bsgamma,r'}(\calS) \le
e_{s,\alpha,\bsgamma,1}(\calS)$. On the other hand, for any $q$ and $r$ we
have $\|F\|_{s,\alpha,\bsgamma,q,r} \ge
\|F\|_{s,\alpha,\bsgamma,q,\infty}$. Hence by restricting ourselves to the
case $r=\infty$, we are working with the larger quantity
$e_{s,\alpha,\bsgamma,1}(\calS)$, but we benefit from having a smaller
norm $\|F\|_{s,\alpha,\bsgamma,q,\infty}$. This is the main reason why we
are able to obtain an improved convergence rate compared with other
papers.

\textbf{SPOD weights.} For a function $F$ satisfying \eqref{eq:like-norm},
its norm, with $r=\infty$ and any $q$, can be bounded~by
\begin{align*}
 \|F\|_{s,\alpha,\bsgamma,q,\infty}
 &\,\le\, c
 \max_{\setu\subseteq\{1:s\}}  \gamma_\setu^{-1}
 \sum_{\setv\subseteq\setu} \sum_{\bstau_{\setu\setminus\setv} \in \{1:\alpha\}^{|\setu\setminus\setv|}}\!
 |(\bsalpha_\setv,\bstau_{\setu\setminus\setv},\bszero)|!\,
 \bsbeta_s^{(\bsalpha_\setv,\bstau_{\setu\setminus\setv},\bszero)}
 \\
 &\,=\, c
 \max_{\setu\subseteq\{1:s\}}
 \gamma_\setu^{-1}
 \sum_{\bsnu_\setu \in \{1:\alpha\}^{|\setu|}}
 |\bsnu_\setu|!\,\prod_{j\in\setu} \left(2^{\delta(\nu_j,\alpha)}\beta_j^{\nu_j}\right)\;,
\end{align*}
where $\delta(\nu_j,\alpha)$ is $1$ if $\nu_j=\alpha$ and is $0$
otherwise.
To make $\|F\|_{s,\alpha,\bsgamma,q,\infty} \le c$, we choose
\begin{equation} \label{eq_defSPOD}
 \gamma_\setu
 \,:=\,
 \sum_{\bsnu_\setu \in \{1:\alpha\}^{|\setu|}}
 |\bsnu_\setu|!\,\prod_{j\in\setu} \left(2^{\delta(\nu_j,\alpha)}\beta_j^{\nu_j}\right)\;.
\end{equation}
We shall refer to our new form of weights \eqref{eq_defSPOD} as
``\emph{smoothness-driven product and order dependent weights}", or
``\emph{SPOD weights}'' for short. This new form of weights has similar
characteristics to POD weights---product and order dependent weights,
which were first introduced in \cite{KSS12} in the analysis of QMC methods
for PDEs with random coefficients.

\textbf{Product weights.} In a similar way, we deduce that for a function
$F$ satisfying \eqref{eq:like-norm2} we should choose
\begin{equation} \label{eq_defPROD}
 \gamma_\setu
 \,:=\,
 \sum_{\bsnu_\setu \in \{1:\alpha\}^{|\setu|}}
 \bsnu_\setu!\,\prod_{j\in\setu} \left(2^{\delta(\nu_j,\alpha)}\beta_j^{\nu_j}\right)
 \,=\, \prod_{j\in\setu}
 \sum_{\nu=1}^\alpha \left(\nu!\,2^{\delta(\nu,\alpha)}\beta_j^{\nu}\right)
 \;,
\end{equation}
which is of product form.

Since the quantity $e_{s,\alpha,\bsgamma,r'}(\calS)$ generalizes the
criterion used in \cite{BDGP11, BDLNP12}, the results in \cite{BDGP11,
BDLNP12} can potentially be adapted to show that a ``higher order
polynomial lattice rule'' (a kind of digital net), can be
constructed using a component-by-component (CBC) algorithm to achieve the
convergence rate, with $N=b^m$,
\begin{equation} \label{B-rate}
  e_{s,\alpha,\bsgamma,r'}(\calS) \,=\, \calO(N^{-\tau})
  \quad\mbox{for any}\quad \tau\in [1,\alpha)\;,
\end{equation}
where the implied constant depends
on $\tau$, $\bsgamma$ and $s$. This approach works for product
weights, but for general weights this should be understood only as an
\emph{existence} result, since the CBC construction for general weights is
prohibitively expensive. Furthermore, the construction for such a rule of
order $\alpha$ has a cost which scales with $N^\alpha$, see \cite{BDLNP12}, making it harder
to obtain higher order rules. We do not take this approach, but use an
analogous approach to \cite{Go13,DiGo12}.

\subsection{Interlaced polynomial lattice rules}
\label{sec:PolLattDims}

In this subsection we formally introduce (interlaced) polynomial lattice
rules. Polynomial lattice rules were first introduced by Niederreiter, see
\cite{Nie92}. In the following let $b$ be a prime number, $\Z_b$ be the finite field with $b$ elements, $\Z_b[x]$ be the set of all polynomials with coefficients in $\Z_b$ and $\Z_b((x^{-1}))$ be the set of all formal Laurent series $\sum_{\ell = w}^\infty t_\ell x^{-\ell}$, where $w$ is an arbitrary integer and $t_\ell \in \Z_b$ for all $\ell$.

\begin{definition}[Polynomial lattice rules] \label{def_poly_lat}
For a prime $b$ and any $m \in \bbN$, let $P \in \Z_b[x]$ be an
irreducible polynomial with $\deg(P)= m$; this is known as the
\emph{modulus}. For a given dimension $s\geq 1$, select $s$ polynomials
$q_1(x),\ldots,q_s(x)$ from the set
\begin{equation}\label{eq:Gbm}
   \Gc_{b, m} \,:=\, \{q(x) \in \Z_b[x] \setminus \{0\} \,:\, \text{deg}(q) < m\}\;,
\end{equation}
and write collectively
\begin{equation}\label{eq:GenVecq}
 \bsq \,=\, \bsq(x) \,=\, (q_1(x),\ldots,q_s(x)) \in
\Gc^s_{b, m}  \;;
\end{equation}
this is known as the \emph{generating vector}. For each integer $0 \le n <
b^m$, let $n = \eta_0 + \eta_1 b + \cdots + \eta_{m-1} b^{m-1}$ be the
$b$-adic expansion of $n$, and we associate with $n$ the polynomial
\[
  n(x) = \sum_{r=0}^{m-1} \eta_r \, x^r  \in \Z_b[x]\;.
\]
Furthermore, we denote by $v_{m}$ the map from $\Z_b((x^{-1}))$ to the
interval $[0,1)$ defined for any integer $w$ by
\[
  v_{m}\left( \sum_{\ell=w}^\infty t_\ell\, x^{-\ell} \right) =
  \sum_{\ell=\max(1,w)}^m t_\ell \, b^{-\ell}\;.
\]
Then, the QMC point set $\calS_{P,b,m,s}(\bsq)$ of a (classical)
\emph{polynomial lattice rule} comprises the points
\[
  \bsy_n = \left( v_{m} \left( \frac{n(x) q_1(x)}{P(x)} \right),\ldots,
    v_{m} \left( \frac{n(x) q_s(x)}{P(x)} \right) \right) \in [0,1)^s,
    \quad n = 0, \ldots, b^m-1\;.
\]
\end{definition}

In the following we define interlaced polynomial lattice rules \cite{Go13,DiGo12},
belonging to the family of higher order digital nets, which were first
introduced in \cite{D07,D08}.

\begin{definition}[Interlaced polynomial lattice rules] \label{int-poly}
Define the \emph{digit interlacing function}
with interlacing factor $\alpha \in \mathbb{N}$ by
\begin{equation}\label{eq:DigIntl}
\begin{array}{rcl}
\mathscr{D}_\alpha: [0,1)^{\alpha} & \to & [0,1)
\\
(x_1,\ldots, x_{\alpha}) &\mapsto & \sum_{a=1}^\infty \sum_{j=1}^\alpha
\xi_{j,a} b^{-j - (a-1) \alpha}\;,
\end{array}
\end{equation}
where $x_j = \xi_{j,1} b^{-1} + \xi_{j,2} b^{-2} + \cdots$ for $1 \le j
\le \alpha$. We also define such a function for vectors by setting
\begin{equation} \label{eq:DigIntlAr}
\begin{array}{rcl}
\mathscr{D}_\alpha: [0,1)^{\alpha s} & \to & [0,1)^s
\\
(x_1,\ldots, x_{\alpha s})
&\mapsto &
(\mathscr{D}_\alpha(x_1,\ldots, x_\alpha),  \ldots,
\mathscr{D}_\alpha(x_{(s-1)\alpha +1},\ldots, x_{s \alpha}))
\;.
\end{array}
\end{equation}
Then, an \emph{interlaced polynomial lattice rule of order $\alpha$ with
$b^m$ points in $s$ dimensions} is a QMC rule using
$\mathscr{D}_\alpha(\calS_{P,b,m,\alpha s}(\bsq))$ as quadrature points,
for some given modulus $P$ and generating vector
$\bsq \in \Gc^{\alpha s}_{b, m}$.
\end{definition}

Note that the generating vector $\bsq$ is of length $\alpha\, s$, and it
can be interpreted as the concatenation of $\alpha$ different generating
vectors of (classical) polynomial lattice rules in $s$ dimensions.
Interlaced scrambled polynomial lattice rules were first used in
\cite{DiGo12} and interlaced polynomial lattice rules were first used in
\cite{Go13}.

An illustration of how interlacing works is as follows: if we have
$\alpha$ numbers in base $b$ representation
$
  (0.\,\xi_{1,1}\xi_{1,2}\xi_{1,3}\cdots)_b\;, 
  (0.\,\xi_{2,1}\xi_{2,2}\xi_{2,3}\cdots)_b\;, 
  \ldots
  \;,
  (0.\,\xi_{\alpha,1}\xi_{\alpha,2}\xi_{\alpha,3}\cdots)_b\;,
$
then the result of interlacing is $
  (0.\,\xi_{1,1}\xi_{2,1}\cdots\xi_{\alpha,1}\,
       \xi_{1,2}\xi_{2,2}\cdots\xi_{\alpha,2}\,
       \xi_{1,3}\xi_{2,3}\cdots\xi_{\alpha,3}\cdots)_b\;,
$ 
that is, we take the first digit of all the $\alpha$ numbers, followed
by the second digit of all the $\alpha$ numbers, and then the third digit,
and so on.

\begin{remark}[Higher order polynomial lattice rules]
Higher order polynomial lattice rules were first introduced in
\cite{DiPi07}. It is important to note that they are \emph{not} obtained
by interlacing (classical) polynomial lattice rules. A higher-order
polynomial lattice rule of order $\alpha$ differs from
Definition~\ref{def_poly_lat} in two ways: the modulus $P$ has a higher
degree $\deg(P) = \alpha\,m$, and the set of polynomials \eqref{eq:Gbm}
now include all polynomials with $\deg(q)<\alpha\,m$. An interlaced
polynomial lattice rule as defined in Definition~\ref{int-poly} is
\emph{not} a higher order polynomial lattice rule, however, they still belong to the family of higher order digital nets and therefore achieve a higher order of convergence of the integration error of smooth functions, which is the property used in this paper.
\end{remark}

To proceed with our analysis, we note that the worst case error bound in
Theorem~\ref{thm:wce} holds for interlaced polynomial lattice rules, but
the dual net \eqref{dual} is expressed in terms of the generating
matrices, which is inconvenient for our analysis or computation. We now
derive an alternative expression for $e_{s,\alpha,\bsgamma,r'}(\calS)$ in
\eqref{def-B} which is expressed in terms of the generating polynomials.

We start by extending the definition of the interlacing function
$\mathscr{D}_\alpha$ to nonnegative integers by setting
\begin{equation}\label{eq_DefInterl}
\begin{array}{rcl}
 \mathscr{E}_\alpha: \mathbb{N}^{\alpha}_0 & \to & \mathbb{N}_0 \\
 (\ell_1,\ldots, \ell_{\alpha}) & \mapsto &
 \sum_{a=0}^\infty \sum_{j=1}^\alpha l_{j,a}\, b^{j -1 + a \alpha}
\;,
\end{array}
\end{equation}
where $\ell_j = l_{j,0} + l_{j,1} b + l_{j,2} b^2 + \cdots$ for $1 \le j
\le \alpha$. We also extend this function to vectors via
\begin{equation}\label{eq_DefInterlV}
\begin{array}{rcl}
\mathscr{E}_\alpha: \mathbb{N}^{\alpha s}_0 & \to & \mathbb{N}^s_0
\\
(\ell_1,\ldots, \ell_{\alpha s}) & \mapsto &
(\mathscr{E}_\alpha(\ell_1,\ldots, \ell_\alpha), \ldots,
\mathscr{E}_\alpha(\ell_{\alpha (s-1)+1}, \ldots, \ell_{\alpha s}))
\;.
\end{array}
\end{equation}
For a given set $\emptyset\ne\setv\subseteq\{1:\alpha s\}$, we define
\begin{equation} \label{u_of_v}
  \setu(\setv) \,:=\, \{\lceil j/\alpha \rceil: j \in \setv\} \,\subseteq\, \{1:s\}\;,
\end{equation}
where each element appears only once as is typical for sets. The set
$\setu(\setv)$ can be viewed as an indicator on whether the set $\setv$
includes any element from each block of $\alpha$ components from
$\{1:\alpha s\}$.

The dual net can be obtained by interlacing the dual nets corresponding to the
$\alpha$ different (classical) polynomial lattice rules. By
\cite[Lemma~10.6]{DiPi10} and the definition \eqref{eq_DefInterlV} of the
interlacing function, we can rewrite \eqref{def-B} for an interlaced
polynomial lattice rule with $r'=1$ as
\begin{equation} \label{eq:error-1}
  e_{s,\alpha,\bsgamma,1}(\calS)
  \, = \,
  \sum_{\emptyset\neq \setv \subseteq\{1:\alpha s\}}
  C_{\alpha,b}^{|\setu(\setv)|}\, \gamma_{\setu(\setv)}
  \sum_{\bsell_\setv \in \calD_\setv^*}b^{-\mu_\alpha(\mathscr{E}_\alpha(\bsell_{\setv},\bszero))}
  \;,
\end{equation}
where $(\bsell_\setv,\bszero)$ denotes a vector of length $\alpha s$ whose
$j$th component is $\ell_j$ if $j\in\setv$ and $0$ if $j\notin\setv$, and
the ``dual net without $0$ components''
is now defined in terms of the generating polynomials as
\begin{equation} \label{dual-poly}
  \calD_\setv^*
  \,:=\,
  \{\bsell_\setv \in \mathbb{N}^{|\setv|}: \tru_m (\bsell_\setv) \cdot \bsq_\setv \equiv 0 \pmod{P} \}\;.
\end{equation}
The ``inner-product'' in \eqref{dual-poly} denotes
\[
  \tru_m(\bsk_\setv) \cdot \bsq_\setv
  \,=\, \sum_{j\in\setv} \tru_m(k_j)(x)\, q_j(x) \,\in\, \Z_b[x]\;,
\]
with the polynomial $\tru_m(k)(x)$ defined as follows: to any nonnegative
integer $k$ with $b$-adic expansion $k = \kappa_0 + \kappa_1 b + \cdots +
\kappa_{\rho-1} b^{\rho-1}$, we associate a unique polynomial $k(x) =
\kappa_0 + \kappa_1 x + \cdots + \kappa_{\rho-1} x^{\rho-1}$ and its
truncated version
\[
  \tru_m(k)(x) \,=\, \kappa_0 + \kappa_1 x + \cdots + \kappa_{m-1} x^{m-1}\;,
\]
which is obtained by setting $\kappa_\rho = \cdots = \kappa_{m-1} = 0$ if
$\rho<m$.

\subsection{Component-by-component construction}
\label{cbc-constr}

In this subsection we introduce and analyze a component-by-component (CBC)
construction of interlaced polynomial lattice rules. The expression
\eqref{eq:error-1} could be used as our search criterion. However, to
reduce the computational cost from a scaling of $N^\alpha$ to $\alpha N$,
we shall instead work with an upper bound to \eqref{eq:error-1}, which is based on the following lemma.

\begin{lemma}
For any $z_1,\ldots,z_\alpha\in\bbN_0$ we have
\begin{align}\label{eq:mualpha-to-mu1s}
  \mu_\alpha(\mathscr{E}_\alpha(z_1, \ldots, z_\alpha))
  \,\ge\,
  \alpha \sum_{j=1}^\alpha \mu_1(z_j) - \frac{\alpha(\alpha-1)}{2}\;.
\end{align}
\end{lemma}

\begin{proof}
For $\alpha = 1$ we obviously have equality. For $\alpha \ge 2$, we define
$z'_j$ by retaining only the most significant base~$b$ digit in~$z_j$, and
we set $z'_j=0$ if $z_j = 0$. Then we have $\mu_1(z'_j) = \mu_1(z_j)$ and
$z'_j \ne 0 \Leftrightarrow z_j \ne 0$, and thus
\begin{align*}
  \mu_\alpha(\mathscr{E}_\alpha(z_1, \ldots, z_\alpha))
  &\,\ge\, \mu_\alpha(\mathscr{E}_\alpha(z'_1, \ldots, z'_\alpha))
  \,=\, \sum_{\substack{j=1 \\ z_j \ne 0}}^\alpha \left( \alpha \left(\mu_1(z_j) - 1 \right) + j \right) \\
  &\,=\,
    \alpha \sum_{j=1}^\alpha \mu_1(z_j)
    -
    \sum_{\substack{j=1 \\ z_j \ne 0}}^\alpha (\alpha - j)
    \,\ge\,
    \alpha \sum_{j=1}^\alpha \mu_1(z_j)
    -  \sum_{j=1}^\alpha (\alpha - j)\;,
\end{align*}
which proves the result. \qquad\end{proof}

We now apply \eqref{eq:mualpha-to-mu1s} to the vector
$(\bsell_\setv,\bszero)$ in blocks of $\alpha$ components, noting that
$\mu_\alpha(\mathscr{E}_\alpha(\bszero)) = 0$, to obtain
\[
  \mu_\alpha (\mathscr{E}_\alpha(\bsell_\setv,\bszero))
  \,\ge\, \alpha \mu_1 (\bsell_\setv) - \frac{\alpha(\alpha-1)}{2} |\setu(\setv)|\;.
\]
Substituting this into \eqref{eq:error-1} then yields the upper bound
\begin{equation} \label{eq:error-2}
  e_{s,\alpha,\bsgamma,1}(\calS)
  \,\le\,
  \sum_{\emptyset\neq \setv \subseteq\{1:\alpha s\}}
  C_{\alpha,b}^{|\setu(\setv)|}\, \gamma_{\setu(\setv)}\,b^{\alpha(\alpha-1)|\setu(\setv)|/2}
  \sum_{\bsell_\setv \in \calD_\setv^*}b^{-\alpha\mu_1(\bsell_\setv)}
  \;.
\end{equation}

We shall use the right-hand side of \eqref{eq:error-2} as our search
criterion in the CBC construction. To simplify our notation, we define
\begin{equation} \label{eq:def-E}
  E_d(\bsq) \,:=\,
  \sum_{\emptyset \neq \setv \subseteq \{1:d\}} \widetilde{\gamma}_{\setv}
  \sum_{\bsell_\setv \in \calD_\setv^*}
  b^{- \alpha \mu_1(\bsell_\setv)}
\;.
\end{equation}
In particular, we are interested in the case $d = \alpha s$ and weights
\begin{equation} \label{eq:weight2}
  \widetilde{\gamma}_{\setv} \,:=\,
  C_{\alpha,b}^{|\setu(\setv)|}\, \gamma_{\setu(\setv)}\,b^{\alpha(\alpha-1)|\setu(\setv)|/2}\;.
\end{equation}
However, the theorem below holds for any $d$ and general weights
$\widetilde{\gamma}_{\setv}$.

\begin{theorem}[CBC error bound]\label{thm_cbc}
Let $b \ge 2$ be prime, and $\alpha \ge 2$ and $m,d \ge 1$ be integers,
and let $P \in \Z_b[x]$ be an irreducible polynomial with $\deg(P) = m$.
Let $(\widetilde{\gamma}_{\setv})_{\setv\subseteq\{1:d\}}$ be positive
real numbers. Then a generating vector $\bsq^* = (1, q_2^*, 
\ldots,
q_d^*) \in \Gc^d_{b,m}$ can be constructed using a component-by-component
approach, minimizing $E_d(\bsq)$ in each step, such that
\begin{equation} \label{cond1}
 E_d(\bsq^*) \,\le\,
 \Bigg( \frac{2}{b^{m}-1}\sum_{\emptyset\ne\setv\subseteq\{1:d\}}
         \widetilde\gamma_{\setv}^{\lambda} \left(\frac{b-1}{b^{\alpha\lambda}-b}\right)^{|\setv|}
 \Bigg)^{1/\lambda}
 \quad\mbox{for all}\quad \lambda \in (1/\alpha, 1]
 \;.
\end{equation}
\end{theorem}

\begin{proof}
We prove the result by induction. For $d=1$, we calculate
\begin{align*}
 E_1(1)
 \,=\, \widetilde\gamma_{\{1\}} \sum_{\ell=1}^\infty b^{-\alpha \mu_1(\ell\, b^m)}
 \,=\, \widetilde\gamma_{\{1\}} \sum_{a=0}^\infty b^{-\alpha (m + a +1 )} (b-1) b^a
 &\,=\, \widetilde\gamma_{\{1\}} b^{-\alpha m} \frac{b-1}{b^\alpha -b}
\;.
\end{align*}
Thus the result holds for $d=1$.

Suppose now that \eqref{cond1} holds for some vector
$\bsq^* \in \Gc_{b,m}^d$ for some $d \geq 1$.
For $\bsk \in \mathbb{N}_0^d$ we denote the support of the
multi-index $\bsk$ by $\setv(\bsk) := \{1 \le j \le d: k_j > 0\}$. Then we
write \eqref{eq:def-E}, with $d$ replaced by $d+1$, in an equivalent
formulation
\begin{align}\label{cond2}
 E_{d+1}(\bsq^*, q_{d+1})
 &\,=\, \sum_{\satop{(\bsk,k_{d+1}) \in \N_0^{d+1}\setminus \{{\bf 0}\}}
 {\tru_m(\bsk,k_{d+1})\cdot (\bsq^*,q_{d+1}) \equiv 0 \pmod{P}}}
 \widetilde\gamma_{\setv(\bsk,k_{d+1})}\, b^{-\alpha \mu_1(\bsk,k_{d+1})}
 \nonumber \\
 &\,=\, E_d(\bsq^*) + \theta(q_{d+1})\;,
\end{align}
where we have separated out the $k_{d+1}=0$ terms so that
\[
  \theta(q_{d+1})
= \sum_{k_{d+1}=1}^\infty \Bigg( b^{-\alpha \mu_1(k_{d+1})}
    \sum_{\satop{\bsk \in \N_0^d}
                {\tru_m(\bsk)\cdot \bsq \equiv -\tru_m(k_{d+1}) \cdot q_{d+1} \pmod{P}}}
    \widetilde\gamma_{\setv(\bsk,k_{d+1})}\, b^{-\alpha \mu_1(\bsk)} \Bigg)\;.
\]
By the induction assumption, we may assume that a minimizer $\bsq^* =
(q^*_1,...,q^*_d)$ of $E_d(\bsq^*)$ in \eqref{cond2} has already been
determined. Then, the CBC algorithm chooses $q_{d+1}^*$ such that
$E_{d+1}(\bsq^*, q_{d+1})$ is minimized. By \eqref{cond2}, the only
dependency on $q_{d+1}$ of $ E_{d+1}(\bsq^*, q_{d+1})$ enters via
$\theta(q_{d+1})$. Therefore, we conclude
$\theta(q_{d+1}^*) \le \theta(q_{d+1})$ for all $q_{d+1} \in \Gc_{b,m}$,
which implies that for any $\lambda \in (0,1]$ we have
$[\theta(q_{d + 1}^*)]^{\lambda} \le [\theta(q_{d+1})]^{\lambda} $ for all $q_{d+1} \in \Gc_{b,m}$.
Since the number of elements in $\Gc_{b,m}$ is $b^m-1$ and since $[\theta(q_{d+1})]^{\lambda}$ is bounded by the average over all $[\theta(q_{d+1})]^{\lambda}$, we obtain
\begin{equation} \label{bound1}
    \theta(q_{d+1}^*)
\le \Bigg(\frac{1}{b^m-1} \sum_{q_{d+1} \in \Gc_{b,m}} [\theta(q_{d+1})]^{\lambda} \Bigg)^{1/\lambda}.
\end{equation}
We will obtain a bound on $\theta(q_{d+1}^*)$ through this last
inequality.

Let $\lambda \in (1/\alpha,1]$. It follows from Jensen's inequality\footnote{$(\sum_k a_k)^\lambda \le \sum_k a_k^\lambda$ for $0 < \lambda \le 1$ and $a_k \ge 0$.} that
\[
    [\theta(q_{d+1})]^{\lambda}
\,\le\, \sum_{k_{d+1}=1}^\infty b^{-\alpha\lambda\,\mu_1(k_{d+1})}
    \sum_{\satop{\bsk \in \N_0^d}
  {\tru_m(\bsk)\cdot \bsq \equiv -\tru_m(k_{d+1}) \cdot q_{d+1} \pmod{P}}}
  \widetilde\gamma_{\setv(\bsk,k_{d+1})}^{\lambda} b^{-\alpha\lambda\,\mu_1(\bsk)}.
\]
If $k_{d+1}$ is a multiple of $b^m$, then $\tru_m(k_{d+1}) = 0$ and the
corresponding term in the sum is independent of $q_{d+1}$. If $k_{d+1}$ is
not a multiple of $b^m$, then $\tru_m(k_{d+1})$ is a non-zero polynomial of degree less than $m$. Moreover, since $q_{d+1} \ne 0$ and $P$ is irreducible,
$\tru_m(k_{d+1}) \cdot q_{d+1}$ is never a multiple of $P$. Thus
\begin{align} \label{bound2}
&\frac{1}{b^m-1} \sum_{q_{d+1} \in \Gc_{b,m}} [\theta(q_{d+1})]^{\lambda}
\le \sum_{\satop{k_{d+1}=1}{b^m | k_{d+1}}}^\infty  b^{-\alpha\lambda\,\mu_1(k_{d+1})}
  \!\!\!\!\!\!  \sum_{\satop{\bsk \in \N_0^d}
  {\tru_m(\bsk)\cdot \bsq \equiv 0 \pmod{P}}} \!\!\!\!\!\!
  \widetilde\gamma_{\setv(\bsk,k_{d+1})}^{\lambda}\,
  b^{-\alpha\lambda\,\mu_1(\bsk)}
\nonumber\\
&\qquad\qquad\qquad + \frac{1}{b^m-1} \sum_{\satop{k_{d+1}=1}{b^m \nmid\, k_{d+1}}}^\infty  b^{-\alpha\lambda\,\mu_1(k_{d+1})}
    \sum_{\satop{\bsk \in \N_0^d}
  {\tru_m(\bsk)\cdot \bsq \not\equiv 0 \pmod{P}}} \widetilde\gamma_{\setv(\bsk,k_{d+1})}^{\lambda}\,
  b^{-\alpha\lambda\,\mu_1(\bsk)}
\nonumber\\
&\,\le\,
\left(b^{-\alpha\lambda m} + \frac{1}{b^m-1} \right)
\sum_{d+1 \in \setv \subseteq \{1:d+1\}}
\widetilde\gamma_\setv^{\lambda} \left(\frac{b-1}{b^{\alpha\lambda}-b} \right)^{|\setv|},
\end{align}
where we used the following estimates
\begin{align*}
&      \sum_{\satop{k_{d+1}=1}{b^m | k_{d+1}}}^\infty b^{-\alpha\lambda\,\mu_1(k_{d+1})}
\,=\,  b^{-\alpha\lambda (m+1)} \frac{b^{\alpha\lambda} (b-1)}{b^{\alpha\lambda} -b}, \quad
      \sum_{\satop{k_{d+1}=1}{b^m \nmid\, k_{d+1}}}^\infty b^{-\alpha\lambda\,\mu_1(k_{d+1})}
\,\le\, \frac{b-1}{b^{\alpha\lambda}-b}, \\
&\qquad\quad \sum_{\satop{\bsk \in \N_0^d}{\tru_m(\bsk)\cdot \bsq \not\equiv 0 \pmod{P}}}
      \widetilde\gamma^{\lambda}_{\setv(\bsk)}\, b^{-\alpha\lambda\,\mu_1(\bsk)}
\,\le\, \sum_{\setv \subseteq \{1:s\}} \widetilde\gamma_\setv^{\lambda}
\left(\frac{b-1}{b^{\alpha\lambda}-b}\right)^{|\setv|}.
\end{align*}
Hence we have from \eqref{bound1} and \eqref{bound2} that
\[
 \theta(q_{d+1}^*)
 \,\le\, \Bigg( \frac{2}{b^m-1} \sum_{d+1 \in \setv \subseteq \{1:d+1\}}
 \widetilde\gamma_{\setv}^{\lambda} \left(\frac{b-1}{b^{\alpha\lambda}-b} \right)^{|\setv|}
 \Bigg)^{1/\lambda}\;,
\]
which, together with \eqref{cond1} and \eqref{cond2}, yields the required
estimate for $E_{d+1}(\bsq^*,q_{d+1}^*)$, that is, \eqref{cond1} with $d$
replaced by $d+1$. This completes the proof. 
\qquad\end{proof}

\begin{theorem}[CBC error bound] \label{thm:wce2}
Let $\alpha,s\in\bbN$ with $\alpha>1$, $1\le q\le\infty$, 
and let $\bsgamma = (\gamma_\setu)_{\setu\subset\bbN}$ denote a collection
of weights. Let $b$ be prime and let $m\in\bbN$ be arbitrary. Then, an
interlaced polynomial lattice rule of order $\alpha$ with $N=b^m$ points
$\{\bsy_0,\ldots,\bsy_{n-1}\}\in [0,1]^s$ can be constructed using a
component-by-component algorithm, such that
\begin{align*}
  \left| \frac{1}{b^m} \sum_{n=0}^{b^m-1} F(\bsy_n) - \int_{[0,1]^s} F(\bsy) \,\rd \bsy\right|
  \le \!\!
  \left(\frac{2}{b^{m}-1} \sum_{\emptyset\ne\setu\subseteq{\{1:s\}}}
  \!\!\gamma_\setu^\lambda\, [\rho_{\alpha,b}(\lambda)]^{|\setu|}\right)^{\!\!1/\lambda}\!\!
  \|F\|_{s,\alpha,\bsgamma,q,\infty},
\end{align*}
for all $1/\alpha < \lambda \le 1$, where
\begin{equation}\label{eq:defrhoab}
  \rho_{\alpha,b}(\lambda) \,:=\,
  \left(C_{\alpha,b}\,b^{\alpha(\alpha-1)/2}\right)^\lambda
  \left(\left(1+\frac{b-1}{b^{\alpha\lambda}-b}\right)^\alpha-1\right)\;,
\end{equation}
with $C_{\alpha,b}$ defined by \eqref{eq:Cab}.
\end{theorem}

\begin{proof}
Theorem~\ref{thm_cbc} states that an interlaced polynomial lattice rule
with interlacing factor $\alpha$ in $s$ dimensions can be constructed by
the CBC algorithm, with weights \eqref{eq:weight2}, such that
\begin{align*}
  &e_{\alpha,\bsgamma,s,1}(\calS)
  \,\le\, E_{\alpha s}(\bsq^*) \\
  &\,\le\, \Bigg( \frac{2}{b^{m}-1}
  \sum_{\emptyset\ne\setv\subseteq\{1:\alpha s\}}
  \left(C_{\alpha,b}^{|\setu(\setv)|}\, \gamma_{\setu(\setv)}\, b^{\alpha(\alpha-1)|\setu(\setv)|/2}\right)^\lambda
  \left(\frac{b-1}{b^{\alpha\lambda}-b} \right)^{|\setv|}
  \Bigg)^{1/\lambda} \\
  &\,=\, \Bigg( \frac{2}{b^{m}-1}
  \sum_{\emptyset\ne\setu\subseteq\{1:s\}}
  \left(C_{\alpha,b}^{|\setu|}\, \gamma_{\setu}\, b^{\alpha(\alpha-1)|\setu|/2}\right)^\lambda
  \left(\left(1+\frac{b-1}{b^{\alpha\lambda}-b} \right)^\alpha-1\right)^{|\setu|}
  \Bigg)^{1/\lambda}\;. 
\end{align*}
This yields the error bound in the theorem.
\end{proof}

In the following we discuss the two specific choices of weights
\eqref{eq_defSPOD} and \eqref{eq_defPROD}.

\textbf{SPOD weights.} Substituting in our choice of $\gamma_\setu$ from
\eqref{eq_defSPOD} and applying Jensen's inequality, we obtain
\begin{align} \label{eq:want}
  e_{\alpha,\bsgamma,s,1}(\calS)
  &\,\le\, \Bigg( \frac{2}{b^{m}-1}
  \sum_{\emptyset\ne\setu\subseteq\{1:s\}}
  \sum_{\bsnu_\setu \in \{1:\alpha\}^{|\setu|}}
  (|\bsnu_\setu|!)^\lambda \,\prod_{j\in\setu}
  \left(B\,2^{\delta(\nu_j,\alpha)} \beta_j^{\nu_j}\right)^\lambda
  \Bigg)^{1/\lambda} \nonumber\\
  &\,=\, \Bigg( \frac{2}{b^{m}-1}
  \sum_{\bszero\ne\bsnu\in\{0:\alpha\}^s}
  (|\bsnu|!)^\lambda \,\prod_{\satop{j=1}{\nu_j>0}}^s
  \left(B\,2^{\delta(\nu_j,\alpha)} \beta_j^{\nu_j}\right)^\lambda
  \Bigg)^{1/\lambda}\;,
\end{align}
where
\begin{align} \label{eq:B}
  B \,:=\, C_{\alpha,b}\, b^{\alpha(\alpha-1)/2}
  \left(\left(1+\frac{b-1}{b^{\alpha\lambda}-b} \right)^\alpha-1\right)^{1/\lambda}\;.
\end{align}
We now choose $\lambda$ to ensure that the sum in \eqref{eq:want} is
bounded independently of $s$. Let $\widetilde\beta_j :=
2\max(B,1)\beta_j$. Then the sum in \eqref{eq:want} is bounded by
\[
  \sum_{\bszero\ne\bsnu\in\{0:\alpha\}^s}
  \bigg(|\bsnu|! \,\prod_{j=1}^s
  \widetilde\beta_j^{\nu_j} \bigg)^\lambda\;,
\]
where each term in the sum to be raised to the power of $\lambda$ is of
the form
\begin{equation} \label{eq:form}
  (\nu_1 + \nu_2 + \cdots + \nu_s)!\,
  \underbrace{\widetilde\beta_1 \cdots \widetilde\beta_1}_{\nu_1}
  \underbrace{\widetilde\beta_2 \cdots \widetilde\beta_2}_{\nu_2}
  \cdots
  \underbrace{\widetilde\beta_s \cdots \widetilde\beta_s}_{\nu_s}\;.
\end{equation}
We now define a sequence $d_j := \widetilde\beta_{\lceil j/\alpha\rceil}$
so that $d_1 = \cdots = d_\alpha = \widetilde\beta_1$ and $d_{\alpha+1} =
\cdots = d_{2\alpha} = \widetilde\beta_2$, and so on. Then any term of the
form \eqref{eq:form} can be written as
$
  |\setv|!\, \prod_{j\in\setv} d_j
$
for some subset of indices $\setv\subset\bbN$. Thus we conclude that
\begin{align} \label{eq:last}
  \sum_{\bszero\ne\bsnu\in\{0:\alpha\}^s}
  \bigg(|\bsnu|! \,\prod_{\satop{j=1}{\nu_j>0}}^s
  \widetilde\beta_j^{\nu_j} \bigg)^\lambda
  &\,<\, \sum_{\satop{\setv\subset\bbN}{|\setv|<\infty}}
  \bigg(|\setv|!\, \prod_{j\in\setv} d_j\bigg)^\lambda \nonumber\\
  &\,=\, \sum_{\ell=0}^\infty (\ell!)^\lambda
  \sum_{\satop{\setv\subset\bbN}{|\setv|=\ell}} \prod_{j\in\setv} d_j^\lambda
  \,\le\, \sum_{\ell=0}^\infty (\ell!)^{\lambda-1}
  \bigg(\sum_{j=1}^\infty d_j^\lambda\bigg)^\ell\;.
\end{align}
Note that $\sum_{j=1}^\infty \beta_j^p < \infty$ holds if and only if
$\sum_{j=1}^\infty d_j^p < \infty$. By the ratio test, the last expression
in \eqref{eq:last} is finite if $p\le \lambda < 1$. Alternatively,
using the geometric series formula, the last expression in \eqref{eq:last}
is finite if $\lambda=1$ and $\sum_{j=1}^\infty d_j < 1$. Recall that
$\lambda$ also needs to satisfy $1/\alpha < \lambda\le 1$. Hence we take
\[
  \lambda \,=\, p \qquad\mbox{and}\qquad \alpha \,=\, \lfloor 1/p\rfloor +1\;,
\]
and for
$p=1$ we assume additionally that $\sum_{j=1}^\infty d_j
< 1$, which is equivalent to
\begin{equation} \label{eq:small}
  \sum_{j=1}^\infty \beta_j
  < \frac{1}{2\alpha\max(B,1)}
  = \bigg(4
  \max\bigg(\frac{2}{(2\sin\frac{\pi}{b})^2},\frac{1}{2\sin\frac{\pi}{b}}\bigg)
  \bigg(3+\frac{2}{b}+\frac{2b+1}{b-1}\bigg)
  \bigg(2+\frac{1}{b}\bigg)\bigg)^{-1}.
\end{equation}
Thus we obtain $\calO(N^{-1/p})$ convergence, with the implied constant independent
of~$s$.

\textbf{Product weights.} With the product weights given by \eqref{eq_defPROD}, we obtain
\begin{align*}
  e_{\alpha,\bsgamma,s,1}(\calS)
  &\,\le\, \Bigg( \frac{2}{b^{m}-1}
  \sum_{\emptyset\ne\setu\subseteq\{1:s\}}
  \prod_{j\in\setu}
  \bigg(B\,\sum_{\nu=1}^\alpha \nu!\, 2^{\delta(\nu,\alpha)} \beta_j^{\nu}\bigg)^\lambda
  \Bigg)^{1/\lambda} \\
  &\,\le\, \Bigg( \frac{2}{b^{m}-1}
  \exp
  \Bigg(\sum_{j=1}^s
  \bigg(B\,\sum_{\nu=1}^\alpha \nu!\, 2^{\delta(\nu,\alpha)} \beta_j^{\nu}\bigg)^\lambda \Bigg)
  \Bigg)^{1/\lambda}\;,
\end{align*}
which is bounded independently of $s$ if $\sum_{j=1}^\infty
\beta_j^\lambda < \infty$, where $B$ is as defined in \eqref{eq:B}.
Hence we take again
\[
  \lambda \,=\, p \qquad\mbox{and}\qquad \alpha \,=\, \lfloor 1/p\rfloor +1\;,
\]
to obtain the convergence rate of $\calO(N^{-1/p})$, with the implied
constant independent of~$s$. Note that the condition \eqref{eq:small} is
not needed for product weights when $p=1$.

\subsection{Component-by-component algorithm} \label{sec:CBC}

We first express $E_d(\bsq)$ in \eqref{eq:def-E} in a more convenient
form for computation. Recall from Definition~\ref{def_poly_lat} that the
$j$-th coordinate of the $n$-th point of the interlaced polynomial lattice
point set is
\begin{equation*}
  y_j^{(n)} \,=\, \upsilon_m\left(\frac{n(x)\, q_j(x)}{P(x)} \right)\;.
\end{equation*}
In the following we write $y_j^{(n)}$ for brevity; note however that
$y_j^{(n)}$ depends on the $j$-th component $q_j$ of the generating
vector.
We have
\begin{equation*}
 \sum_{\bsell_\setv \in \calD_\setv^*} b^{- \alpha \mu_1(\bsell_\setv)}
 \,=\, \frac{1}{b^m} \sum_{n=0}^{b^m-1} \sum_{\bsell \in \mathbb{N}^{|\setv|}} b^{-\alpha\mu_1(\bsl)}
 \wal_{\bsell}(\bsy_{\setv}^{(n)})
 \,=\, \frac{1}{b^m} \sum_{n=0}^{b^m-1} \prod_{j \in \setv} \omega(y_j^{(n)})\;,
\end{equation*}
where $\bsy_{\setv}^{(n)} = (y_j^{(n)})_{j \in \setv}$ is the projection
of $n$-th point $\bsy^{(n)}$ onto the coordinates in $\setv$, 
\begin{align*}
  \omega(y)
  \,=\, \sum_{\ell=1}^\infty b^{-\alpha \mu_1(\ell)} \wal_\ell(y)
  \,=\, \frac{b-1}{b^\alpha-b} - b^{\lfloor \log_b y \rfloor (\alpha-1)}
  \frac{b^\alpha-1}{b^\alpha - b}\;,
\end{align*}
and where for $y=0$ we set $b^{\lfloor \log_b 0\rfloor (\alpha-1)} :=
0$. The last equality can be obtained by multiplying
\cite[Eq.~(2)]{DiPi05} by $b^{-\alpha}$.
Hence we can write
\begin{equation} \label{eq:E-CBC}
  E_d(\bsq) \,=\,
  \frac{1}{b^m} \sum_{n=0}^{b^m-1} \sum_{\emptyset \neq \setv \subseteq \{1:d\}}
  \widetilde\gamma_\setv \prod_{j \in \setv} \omega(y_j^{(n)})\;.
\end{equation}

To facilitate the CBC construction, it is important that $E_d(\bsq)$
allows recursively separating out a term depending only on the highest
dimension. The strategy depends on the form of weights.

\textbf{SPOD weights.} Combining \eqref{eq:weight2} with
\eqref{eq_defSPOD}, we obtain SPOD weights
\[
  \widetilde\gamma_\setv =
 \sum_{\bsnu_{\setu(\setv)} \in \{1:\alpha\}^{|\setu(\setv)|}} |\bsnu_{\setu(\setv)}|!
 \prod_{j \in \setu(\setv)} \gamma_j(\nu_j)\,,
 \quad\mbox{with}\quad
 \gamma_j(\nu_j) :=
 C_{\alpha,b}\, b^{\alpha(\alpha-1)/2}\, 2^{\delta(\nu_j,\alpha)}\beta_j^{\nu_j}.
\]
Substituting this into \eqref{eq:E-CBC} yields
\begin{align*}
  E_d(\bsq)
  &\,=\,
  \frac{1}{b^m} \sum_{n=0}^{b^m-1} \sum_{\emptyset \neq \setv \subseteq \{1:d\}}
  \sum_{\bsnu_{\setu(\setv)} \in \{1:\alpha\}^{|\setu(\setv)|}} |\bsnu_{\setu(\setv)}|!
  \bigg(\prod_{j \in \setu(\setv)} \gamma_j(\nu_j)\bigg)
  \bigg(\prod_{j \in \setv} \omega(y_j^{(n)})\bigg)\;.
\end{align*}
Recall that every block of $\alpha$ components in the generating vector
$\bsq$ yields one component for the interlaced polynomial lattice rule.
For convenience, we replace the index $d$ by a double index $(s,t)$ such
that $s$ is the index for the block and $t$ is the index within the block,
that is, we have
\[
  s \,=\, \lceil d/\alpha\rceil
  \quad\mbox{and}\quad t \,=\, (d-1)\bmod \alpha + 1
  \quad\mbox{such that}\quad d \,=\, \alpha(s-1)+t\;.
\]
We then reorder the sums in $E_d(\bsq)$ according to $\bsnu =
(\nu_1,\ldots, \nu_s) \in \{0:\alpha\}^s$ and $\setv \subseteq \{1:d\}$ so
that the set $\setu(\setv)$ consists of the indices $j$ for which
$\nu_j > 0$.
This yields
\begin{align} \label{eq:Est}
  &E_{s,t}(\bsq)
  \,=\,
  \frac{1}{b^m} \sum_{n=0}^{b^m-1}
  \sum_{\satop{\bsnu \in \{0:\alpha\}^s}{|\bsnu|\ne 0}}
  \sum_{\satop{\setv \subseteq \{1:d\}\text{ s.t.}}{\setu(\setv) = \{1\le j\le s\,:\,\nu_j>0\}}} |\bsnu|!
  \bigg(\prod_{j \in \setu(\setv)} \gamma_j(\nu_j)\bigg)
  \bigg(\prod_{j \in \setv} \omega(y_j^{(n)})\bigg) \nonumber\\
  &\,=\,
  \frac{1}{b^m} \sum_{n=0}^{b^m-1} \sum_{\ell=1}^{\alpha s} \ell!
  \sum_{\satop{\bsnu \in \{0:\alpha\}^s}{|\bsnu|=\ell}}
  \bigg(\prod_{\satop{j=1}{\nu_j>0}}^s \gamma_j(\nu_j)\bigg)
  \sum_{\satop{\setv \subseteq \{1:d\}\text{ s.t.}}{\setu(\setv) = \{1\le j\le s\,:\,\nu_j>0\}}}
  \prod_{j \in \setv} \omega(y_j^{(n)})\;.
\end{align}

When $t = \alpha$, that is, when the final block is complete, we have
\begin{align} \label{eq:U}
 &E_{s,\alpha}(\bsq)
 \,=\,
 \frac{1}{b^m} \sum_{n=0}^{b^m-1}
 \sum_{\ell=1}^{\alpha s}
 \underbrace{
 \ell!
 \sum_{\satop{\bsnu \in \{0:\alpha\}^s }{|\bsnu|=\ell}}
 \prod_{\satop{j=1}{\nu_j>0}}^s \bigg[ \gamma_j(\nu_j)
 \bigg(\prod_{i = 1}^{\alpha} (1+\omega(y_{j,i}^{(n)})) -1 \bigg) \bigg]
 }_{=:\, U_{s,\ell}(n)}
 \;,
\end{align}
where we defined the quantity $U_{s,\ell}(n)$, with $U_{0,\ell}(n):=1$,
$U_{s,0}(n):=0$, and $U_{s,\ell}(n) :=0$ for $\ell>\alpha s$.
When $t < \alpha$, that is, when the final block is incomplete, by
separating out the case $\nu_s = 0$ in \eqref{eq:Est}, we obtain
\begin{align*}
 &E_{s,t}(\bsq)
 \,=\,
 \frac{1}{b^m} \sum_{n=0}^{b^m-1}
 \sum_{\ell=1}^{\alpha (s-1)}
 \ell!
 \sum_{\satop{\bsnu \in \{0:\alpha\}^{s-1} }{|\bsnu|=\ell}}
 \prod_{\satop{j=1}{\nu_j>0}}^{s-1} \bigg[ \gamma_j(\nu_j)
 \bigg(\prod_{i = 1}^{\alpha} (1+\omega(y_{j,i}^{(n)})) -1 \bigg) \bigg]
 \nonumber\\
 &\qquad\qquad\quad +
 \frac{1}{b^m} \sum_{n=0}^{b^m-1}
 \sum_{\ell=1}^{\alpha s}
 \sum_{\nu_s=1}^{\min(\alpha,\ell)}
 \ell!
 \sum_{\satop{\bsnu \in \{0:\alpha\}^{s-1} }{|\bsnu|=\ell-\nu_s}} \Bigg(
 \prod_{\satop{j=1}{\nu_j>0}}^{s-1} \bigg[ \gamma_j(\nu_j)
 \bigg(\prod_{i = 1}^\alpha (1+\omega(y_{j,i}^{(n)})) -1\bigg) \bigg] \nonumber\\
 &\qquad\qquad\qquad\qquad\qquad\qquad\qquad\qquad\qquad\qquad\times
 \gamma_s(\nu_s)
 \bigg(\prod_{i = 1}^t (1+\omega(y_{s,i}^{(n)})) -1\bigg) \Bigg)\;, \nonumber
\end{align*}
and thus
\begin{align}  \label{eq:VWX}
 &E_{s,t}(\bsq)
 \,=\, E_{s-1,\alpha}(\bsq) \\
 &\qquad + \frac{1}{b^m} \sum_{n=0}^{b^m-1}
 \bigg(\underbrace{\prod_{i = 1}^t (1+\omega(y_{s,i}^{(n)}))}_{=:\,V_{s,t}(n)}-1 \bigg)
 \bigg(
 \underbrace{
 \sum_{\ell=1}^{\alpha s}
 \underbrace{
 \sum_{\nu_s=1}^{\min(\alpha,\ell)} \gamma_s(\nu_s)\frac{\ell!}{(\ell-\nu_s)!}\, U_{s-1,\ell-\nu_s}(n)
 }_{=:\, X_{s,\ell}(n)}
 }_{=:\, W_s(n)}
 \bigg)\;, \nonumber
\end{align}
where we defined $V_{s,t}(n)$, $W_{s}(n)$, and $X_{s,\ell}(n)$ as indicated, with
$V_{s,0}(n):=1$.

Note that the polynomial $q_{s,t}$ only appears in the final factor of the
products $V_{s,t}(n)$. In particular, the part of \eqref{eq:VWX} that is
affected by $q_{s,t}$ is
\[
 \sum_{n=1}^{b^m-1}
 \omega(y_{s,t}^{(n)})\, V_{s,t-1}(n)\, W_s(n)\;.
\]
Computing this quantity for every $q_{s,t}\in \Gc_{b,m}$ requires the
matrix-vector multiplication with the matrix
\begin{equation*}
  \bsOmega \,:=\,
  \left[\omega\left(\upsilon_m\left(\frac{n(x) q(x)}{P(x)} \right)\right)
  \right]_{\satop{1\le n\le b^m-1}{q\in \Gc_{b,m}}}
\end{equation*}
and the vector $[V_{s,t-1}(n)\,W_s(n)]_{1\le n\le b^m-1}$. The rows and
columns of this matrix can be permuted to allow the matrix-vector
multiplication to be carried out using the fast Fourier transform, see
\cite{NC06a}, with a cost of $\calO(M\,\log M) = \calO(N\,\log N)$ operations, where $M=b^m-1$ and $N =
b^m$. The strategy is based on the Rader transform, see also \cite[Chapter~10.3]{DiPi10}.

Once $q_{s,t}$ is chosen for dimension $\alpha (s-1) + t$, we update the products $V_{s,t}(n)$ by 
\[
  V_{s,t}(n) \,=\, (1+\omega(y_{s,t}^{(n)}))\, V_{s,t-1}(n)\;.
\]
This requires $\calO(N)$ operations. Once we have completed an entire
block of $\alpha$ dimensions, we need to update the values $U_{s,\ell}(n)$
using
\begin{align*}
 U_{s,\ell}(n)
 &\,=\,
 \ell!
 \sum_{\satop{\bsnu \in \{0:\alpha\}^{s-1} }{|\bsnu|=\ell}}
 \prod_{\satop{j=1}{\nu_j>0}}^{s-1} \bigg[ \gamma_j(\nu_j)
 \bigg(\prod_{i = 1}^{\alpha} (1+\omega(y_{j,i}^{(n)})) -1\bigg) \bigg] \\
 &\qquad
 +
 \ell! 
 \sum_{\nu_s=1}^{\min(\alpha,\ell)} \!\!
 \sum_{\satop{\bsnu \in \{0:\alpha\}^{s-1} }{|\bsnu|=\ell-\nu_s}} \!\!\!
 \Bigg(
 \prod_{\satop{j=1}{\nu_j>0}}^{s-1} \bigg[ \gamma_j(\nu_j)
 \bigg(\prod_{i = 1}^{\alpha} (1+\omega(y_{j,i}^{(n)})) -1\bigg) \bigg] \\
 &\qquad\qquad\qquad\qquad\qquad\qquad\times
 \gamma_s(\nu_s)
 \bigg(\prod_{i = 1}^{\alpha} (1+\omega(y_{s,i}^{(n)}))-1 \bigg) \Bigg)\\
 &\,=\, U_{s-1,\ell}(n) + (V_{s,\alpha}(n) -1)\,X_{s,\ell}(n)\;.
\end{align*}
Given that the quantities $V_{s,\alpha}(n)$ and $X_{s,\ell}(n)$ have been
pre-computed and stored, this update requires $\calO(\alpha\, s N)$
operations. We then need to initialize the products $V_{s+1,0}(n)$ by $1$
with $\calO(N)$ operations, and compute the quantities $W_{s+1}(n)$ and
$X_{s+1,\ell}(n)$ with $\calO(\alpha^2 s N)$ operations, before continuing
the search in the new block.

The 
total computational cost for the CBC construction up to $\alpha s$
dimensions is
\[
  \calO(\alpha\, s\, N\,\log N) \mbox{ search cost, plus }
  \calO(\alpha^2 s^2 N) \mbox{ update cost}.
\]
We need to store the quantities $U_{s,\ell}(n)$, $V_{s,t}(n)$, $W_{s}(n)$,
and $X_{s,\ell}(n)$, which can be overwritten as we increase $s$ and $t$. Hence, the total memory
requirement is $\calO(\alpha\, s\, N)$.

We summarize the algorithm in Pseudocode~1 below where $.*$ means element wise multiplication. Note that $\bsU(\ell)$
for $\ell=0,\ldots, \alpha s_{\max}$, $\bsV$, $\bsW$, $\bsX(\ell)$ for
$\ell=1,\ldots, \alpha s_{\max}$, and $\bsE$ are all vectors of length
$N-1$, while $\bsOmega^{\rm perm}$ denotes the permuted version of the
matrix $\bsOmega$.

\textbf{Product weights.} Combining \eqref{eq:weight2} with
\eqref{eq_defPROD}, we obtain product weights
\[
 \widetilde\gamma_\setv \,=\,
 \prod_{j \in \setu(\setv)} \gamma_j\;,
 \quad\mbox{with}\quad
 \gamma_j \,:=\,
 C_{\alpha,b}\, b^{\alpha(\alpha-1)/2}\, \sum_{\nu=1}^\alpha \nu!\, 2^{\delta(\nu,\alpha)}\beta_j^{\nu}\;.
\]
Substituting this into \eqref{eq:E-CBC} yields
\begin{align*}
  E_d(\bsq)
  &\,=\,
  \frac{1}{b^m} \sum_{n=0}^{b^m-1} \sum_{\emptyset \neq \setv \subseteq \{1:d\}}
  \bigg(\prod_{j \in \setu(\setv)} \gamma_j\bigg)
  \bigg(\prod_{j \in \setv} \omega(y_j^{(n)})\bigg) \\
  &\,=\,
  \frac{1}{b^m} \sum_{n=0}^{b^m-1} \sum_{\emptyset \neq \setu \subseteq \{1:s\}}
  \bigg(\prod_{j \in \setu} \gamma_j\bigg)
   \sum_{\satop{\setv \subseteq \{1:d\}}{\setu(\setv)=\setu}}
  \bigg(\prod_{j \in \setv} \omega(y_j^{(n)})\bigg)\;.
\end{align*}
Replacing $d$ by the double index $(s,t)$ as before, we obtain for
$t=\alpha$ that
\begin{align*}
  E_{s,\alpha}(\bsq)
  &\,=\,
  \frac{1}{b^m} \sum_{n=0}^{b^m-1} \underbrace{\prod_{j=1}^s
  \bigg[1 + \gamma_j \bigg(\prod_{i=1}^\alpha (1+ \omega(y_{j,i}^{(n)}) )- 1\bigg)\bigg] }_{=:\,Y_s(n)} - 1\;,
\end{align*}
where we defined the quantity $Y_s(n)$, with $Y_0(n) :=1$. For $t<\alpha$
we have
\begin{align*}
  E_{s,t}(\bsq)
  &\,=\,
  \frac{1}{b^m} \sum_{n=0}^{b^m-1}
  \bigg[1 + \gamma_s \bigg(\underbrace{\prod_{i=1}^t (1+ \omega(y_{s,i}^{(n)}))}_{=:\, V_{s,t}(n)}-1\bigg)\bigg]
  Y_{s-1}(n)
  - 1\;,
\end{align*}
where $V_{s,t}(n)$ is as defined before. The part of
$E_{s,t}(\bsq)$ that is affected by $q_{s,t}$ is
\[
 \sum_{n=1}^{b^m-1}
 \omega(y_{s,t}^{(n)})\, V_{s,t-1}(n)\, Y_{s-1}(n)\;.
\]
Computing this quantity for every $q_{s,t}\in \Gc_{b,m}$ can be done using
the fast Fourier transform as before, with a cost of $\calO(N\,\log N)$
operations. Once $q_{s,t}$ is chosen for each dimension, we need to update
the products $V_{s,t}(n)$ with $\calO(N)$ operations. Once we have
completed an entire block of $\alpha$ dimensions, we need to update the
products $Y_{s}(n)$, again with $\calO(N)$ operations. Hence the total
computational cost is only $\calO(\alpha\, s\, N\,\log N)$ operations,
with the memory requirement of $\calO(N)$. The algorithm for product
weights is summarized in Pseudocode~2 below.

\begin{algorithm}[t] 
\caption{(Fast CBC implementation for SPOD weights)}
\small
\begin{algorithmic}
  \State $\bsU(0) := \bsone$
  \State $\bsU(1:\alpha\, s_{\max}) := \bszero$
  \For{\bf $s$ from $1$ to $s_{\max}$}
  \State $\bsV := \bsone$
  \Comment{initialize products and sums}
  \State $\bsW := \bszero$
  \For{\bf $\ell$ from $1$ to $\alpha s$}
  \State $\bsX(\ell) := \bszero$
  \For{\bf $\nu$ from $1$ to $\min(\alpha,\ell)$}
  \State $\bsX(\ell) := \bsX(\ell) + \gamma_s(\nu) \displaystyle\frac{\ell!}{(\ell-\nu)!}\, \bsU(\ell - \nu)$
  \EndFor
  \State $\bsW := \bsW + \bsX(\ell)$
  \EndFor
  \For{\bf $t$ from $1$ to $\alpha$}
  \State
  $\bsE := \bsOmega^{{\rm perm}}\,(\bsV .\!* \bsW)$
  \Comment{compute -- { use FFT}}
  \State
  $q_{s,t} := {\rm argmin}_{q\in \Gc_{b,m}} E(q)$
  \Comment{select -- { pick the correct index}}
  \State
  $\bsV := \big(\bsone + \bsOmega^{{\rm perm}}(q_{s,t},:)\big) \,.\!*\,\bsV$
  \Comment{update products}
  \EndFor
  \For{\bf $\ell$ from $1$ to $\alpha s$}
  \Comment{update sums}
  \State $\bsU(\ell) := \bsU(\ell) + (\bsV - \bsone) \,.\!*\, \bsX(\ell)$
  \EndFor
  \EndFor
\end{algorithmic}
\end{algorithm}

\begin{algorithm}[t] 
\caption{(Fast CBC implementation for product weights)}
\small
\begin{algorithmic}
  \State $\bsY := \bsone$
  \For{\bf $s$ from $1$ to $s_{\max}$}
  \State $\bsV := \bsone$
  \For{\bf $t$ from $1$ to $\alpha$}
  \State
  $\bsE := \bsOmega^{{\rm perm}}\,(\bsV .\!* \bsY)$
  \Comment{compute -- { use FFT}}
  \State
  $q_{s,t} := {\rm argmin}_{q\in \Gc_{b,m}} E(q)$
  \Comment{select -- { pick the correct index}}
  \State
  $\bsV := \big(\bsone + \bsOmega^{{\rm perm}}(q_{s,t},:)\big) \,.\!*\,\bsV$
  \Comment{update products}
  \EndFor
  \State $\bsY := (\bsone + \gamma_s (\bsV - \bsone)) \,.\!*\, \bsY$
  \Comment{update products}
  \EndFor
\end{algorithmic}
\end{algorithm}

\section{Combined QMC Petrov-Galerkin Error Bound}
\label{sec:comb-err}

At last we return to the PDE problem where the goal is to approximate the
integral \eqref{eq:int} by a QMC Petrov-Galerkin method
\eqref{eq:qmcG}. With a slight abuse of notation, here we denote
\eqref{eq:qmcG} by $Q_{N,s}(G(u^h_s))$, suppressing from our notation the
translation from
$[0,1]^s$ to $[-\frac{1}{2},\frac{1}{2}]^s$. We write the overall error as
\begin{align} \label{total err}
 &I(G(u)) - Q_{N,s}(G(u^h_s))
\\
 &\,=\, [I(G(u)) - I(G(u_s))] + [I(G(u_s)) - I(G(u_s^h))] + [I(G(u_s^h)) - Q_{N,s}(G(u_s^h))]
\;. \nonumber
\end{align}
The first term in \eqref{total err} is the \emph{dimension truncation
error} which was analyzed in Theorem~\ref{thm:trunc}. The second term in
\eqref{total err} is the \emph{Petrov-Galerkin discretization error}
which can be deduced from Theorem~\ref{thm:FEGconv} by taking the vectors
$\bsy$ with $y_j=0$ for $j>s$. The third term in \eqref{total err} is the
\emph{QMC quadrature error} which can be estimated from
Theorem~\ref{thm:main} by noting that for the integrand $F(\bsy) =
G(u_s^h(\bsy))$ we have
$ 
  |(\partial^\bsnu_\bsy F)(\bsy)|
  \,\le\, \|G(\cdot)\|_{\cX'}\, \|(\partial^\bsnu_\bsy u^h_s)(\bsy)\|_{\cX}\;,
$ 
while recognizing that Theorem~\ref{thm:Dsibound} applies also to the
truncated Petrov-Galerkin solution $u^h_s$. We summarize the combined
error estimate in the following theorem.

\begin{theorem}
Under Assumption~\ref{ass:AssBj} and conditions \eqref{eq:psumpsi0},
\eqref{eq:assW1infty}, \eqref{eq:regG}, \eqref{eq:ordered}, if we
approximate the integral \eqref{eq:int} by \eqref{eq:qmcG} using an
interlaced polynomial lattice rule of order $\alpha = \lfloor 1/p \rfloor
+ 1$ with $N=b^m$ points (with $b$ prime) in $s$ dimensions, combined with
a Petrov-Galerkin method in the domain $D$ with one common subspace
$\cX^h$ with $M_h=\operatorname{dim}(\cX^h)$ degrees of freedom and with
the approximation property \eqref{eq:apprprop} with linear cost
$\calO(M_h)$, then there holds the error bound
\[
 |I(G(u)) - Q_{N,s}(G(u^h_s)) | \,\le\, C
 \left( \kappa(s,N)\,\| f \|_{\cY'}\, \| G(\cdot)\|_{\cX'}
  + h^{\tau}\, \| f \|_{{\cY'_{t}}} \| G(\cdot)\|_{{\cX'_{t'}}}  \right)
  \;,
\]
where $\tau = t+ t'$, $C>0$ is independent of $s$, $h$ and $N$, and
\[
  \kappa(s,N) \,=\,
  \begin{cases}
  s^{-2(1/p-1)} + N^{-1/p} & \mbox{if } p \in (0,1)\;, \\
  ( \sum_{j =s+1}^\infty \beta_j )^2 + N^{-1} & \mbox{if } p = 1\;.
  \end{cases}
\]
\end{theorem}
The cost for the evaluation of $Q_{N,s}(G(u^h_s))$ is $\calO(s NM_h)$
operations. The cost for the CBC construction of the interlaced polynomial
lattice rule is $\calO(\alpha\, s\, N\log N + \alpha^2 s^2 N)$ operations
with SPOD weights, plus $\calO(\alpha\, s\, N)$ memory requirement.

\section*{Acknowledgments}

The authors would like to thank Ian Sloan for helpful comments on the paper.


\end{document}